\documentclass[10pt]{amsart}
\usepackage{amsmath,amssymb,amsthm,graphicx,epstopdf,mathrsfs,url}
\usepackage[usenames,dvipsnames]{color}
\usepackage{dsfont}   %Added by Markus
\definecolor{darkred}{rgb}{0.7,0.1,0.1}
\definecolor{darkblue}{rgb}{0.1,0.1,0.4}
\definecolor{darkgrey}{rgb}{0.5,0.5,0.5}
\usepackage[colorlinks=true,linkcolor=darkred,citecolor=Blue]{hyperref}

\numberwithin{equation}{section}
\theoremstyle{plain}% default
\newtheorem{thm}{Theorem}[section]
\newtheorem{lem}[thm]{Lemma}

\newtheorem{prop}[thm]{Proposition}
\newtheorem{cor}[thm]{Corollary}

\newtheorem{definition}[thm]{Definition}
\theoremstyle{remark}
\newtheorem{remark}[thm]{Remark}

\theoremstyle{plain}

\newcommand{\hyp}[1]{$C^{2}$-hypersurface as in Definition~\ref{definition_hypersurface}}
\DeclareMathOperator\ran{ran}

\newcommand{\dom}{\mathrm{dom}\,}

\begin{document}
\title[]{On the spectral properties of Dirac operators with electrostatic \boldmath{$\delta$}-shell interactions}
\author[J. Behrndt]{Jussi Behrndt}
\address{Institut f\"{u}r Numerische Mathematik\\
Technische Universit\"{a}t Graz\\
 Steyrergasse 30, 8010 Graz, Austria\\
E-mail: {\tt behrndt@tugraz.at}}

\author[P. Exner]{Pavel Exner}
\address{Doppler Institute for Mathematical Physics and Applied Mathematics\\ 
Czech Technical University in Prague\\ B\v{r}ehov\'{a} 7, 11519 Prague, Czech Republic,
{\rm and}
Department of Theoretical Physics\\
Nuclear Physics Institute, Czech Academy of Sciences, 
25068 \v{R}e\v{z}, Czech Republic\\
E-mail: {\tt exner@ujf.cas.cz}
}

\author[M. Holzmann]{Markus Holzmann}
\address{Institut f\"{u}r Numerische Mathematik\\
Technische Universit\"{a}t Graz\\
 Steyrergasse 30, 8010 Graz, Austria\\
E-mail: {\tt holzmann@math.tugraz.at}}

\author[V. Lotoreichik]{Vladimir Lotoreichik}
\address{
Department of Theoretical Physics\\
Nuclear Physics Institute, Czech Academy of Sciences, 
25068 \v{R}e\v{z}, Czech Republic\\
E-mail: {\tt lotoreichik@ujf.cas.cz}
}

\begin{abstract}
In this paper the spectral properties of Dirac operators $A_\eta$
with electrostatic $\delta$-shell interactions of constant strength $\eta $ 
supported on compact smooth surfaces in $\mathbb{R}^3$ are studied.
Making use of boundary triple techniques a Krein type resolvent formula and a Birman-Schwinger principle are obtained. 
With the help of these tools some spectral, scattering, and asymptotic
properties of $A_\eta$ are investigated. In particular, it turns out that the discrete spectrum
of $A_\eta$ inside the gap of the essential spectrum is finite, the difference of the third powers of the resolvents of 
$A_\eta$ and the free Dirac operator $A_0$
is trace class, and in the nonrelativistic limit $A_\eta$ converges  in the norm 
resolvent sense to a Schr\"odinger operator
with an electric $\delta$-potential of strength~$\eta$.
\end{abstract}


\maketitle

\section{Introduction} \label{section_introduction}

%This paper is devoted to the study of Dirac operators with electrostatic 
%$\delta$-shell interactions supported
%on bounded and smooth surfaces in $\mathbb{R}^3$.
Singular $\delta$-interactions are often used as idealized replacements 
for strongly localized electric potentials;
the spectral data, scattering properties, and the location of resonances 
for the original operator can be
deduced then approximately. While Schr\"odinger operators with
$\delta$-interactions supported on manifolds of small co-dimensions were investigated extensively, cf. the
monographs \cite{AGHH05, BK13, EK15}
and the review article \cite{E08},
much less attention was paid to Dirac operators with $\delta$-interactions.
%This article shall be a contribution to this growing field.

Let us choose units such that $\hbar = 1$ and denote the speed of 
light by $c$. It is well-known that
the free Dirac operator 
\begin{equation*}
     A_0 := -i c \sum_{j = 1}^3 \alpha_j \partial_j + m c^2 \beta = -i c 
\alpha \cdot \nabla + m c^2 \beta,
     \qquad \dom A_0 = H^1(\mathbb{R}^3; \mathbb{C}^4),
\end{equation*}
where $m > 0$ and $\alpha = (\alpha_1, \alpha_2, \alpha_3)$ and $\beta$ 
denote the Dirac matrices \eqref{def_Dirac_matrices}, is self-adjoint in $L^2(\mathbb{R}^3; \mathbb{C}^4)$ and
that
\begin{equation*}
 \sigma(A_0)= (-\infty, -m c^2] 
\cup [m c^2, \infty).
\end{equation*}
The free Dirac operator
describes the motion of a spin-$\frac{1}{2}$ particle with mass $m$ 
in vacuum taking relativistic aspects into account;
cf. \cite{T92}. In the following let $\Sigma$ be the boundary of a bounded $C^\infty$-smooth domain 
$\Omega \subset \mathbb R^3$. 
Then the Dirac operator with an electrostatic $\delta$-shell interaction supported 
on $\Sigma$ with
constant interaction strength $\eta \in \mathbb{R}$ is formally given by
\begin{equation*}
   A_\eta = -i c \alpha \cdot \nabla + m c^2 \beta + \eta \delta_\Sigma,
\end{equation*}
where $\delta_\Sigma$ stands for the $\delta$-distribution supported on the surface $\Sigma$ acting as
\begin{equation*}
  \delta_\Sigma f = \frac{1}{2} \big( f_+|_\Sigma + f_-|_\Sigma \big); 
  \quad f_+ = f|_{\Omega}, ~f_-=f|_{\mathbb{R}^3 \setminus \overline{\Omega}}.
\end{equation*}
Note that $A_\eta$ is defined on functions that are weakly differentiable away from $\Sigma$, the 
$\delta$-interaction is then modeled, as usual, by a jump condition for these functions on~$\Sigma$.
It is the main objective of this paper to analyze the properties of Dirac operators with 
electrostatic $\delta$-shell interactions by applying the abstract 
technique of quasi boundary triples and their Weyl functions from extension theory of symmetric operators. Our investigations and some of our
results are inspired by the very recent contributions \cite{AMV14, AMV15, AMV16} in this area.

The mathematical study of Dirac operators with $\delta$-interactions started 
in the 1980s. One dimensional Dirac operators with singular point 
interactions were studied in
\cite{GS87}; cf. also \cite[Appendix~J]{AGHH05}, \cite{CMP13} and the 
references therein, and 
the first mathematically rigorous contribution on a Dirac operator in $\mathbb{R}^3$ with 
a $\delta$-shell interaction
supported on a sphere was \cite{DES89}.
Using a decomposition into spherical harmonics and the results on the 
one dimensional Dirac operator
with singular interactions
self-adjointness of $A_\eta$ and a number of spectral properties were shown.
The interest in the topic arose again with the discovery of a family of 
artificial
materials where the Dirac equation can be approximately deduced from 
Schr\"odinger's equation
\cite{WBSB14}. 
 From a mathematical point of view the investigation of Dirac operators 
with
$\delta$-interactions supported on more general surfaces in $\mathbb{R}^3$
was initiated recently in \cite{AMV14, AMV15, AMV16}.
%where the authors defined these Hamiltonians as self-adjoint operators 
%for interactions supported
%on bounded, closed and smooth surfaces and for constant interaction 
%strengths. 
%They have also derived
%some spectral properties of these operators.

Our motivation is to show how the concept of quasi boundary triples and their Weyl functions can 
be used to
introduce and study Dirac operators with electrostatic $\delta$-shell interactions. 
Quasi boundary triples are a slight generalization of the concept of (ordinary) boundary triples, which 
is a powerful tool in the analysis of self-adjoint extensions of symmetric operators 
\cite{B76, BGP08,DM91,GG91, K75}. Quasi boundary triples
were originally introduced in \cite{BL07} for the study of
elliptic partial differential operators, they were applied in the investigation of 
Schr\"odinger operators
with singular interactions in \cite{BLL13}, and they are easily applicable 
also to Dirac operators since in contrast to form methods no semi-boundedness is required.
In this context let us briefly explain our approach to define the Dirac operator 
$A_\eta$ with an electrostatic
$\delta$-shell interaction. Let $S$ be the restriction of the free Dirac operator $A_0$ to functions 
that vanish at~$\Sigma$ and let $S^*$ be its adjoint.
We then construct an operator $T$ which is dense in~$S^*$ and define
the $\delta$-operators $A_\eta$ as restrictions of $T$ to functions that 
satisfy certain jump conditions on $\Sigma$;
cf. Section \ref{section_def_A_eta} for details. 
For $\eta\not=\pm 2 c$ we conclude the self-adjointness of $A_\eta$ and a Krein type formula relating the
resolvent of $A_\eta$ with the resolvent of the free Dirac operator $A_0$ from the general theory of quasi boundary triples
and their Weyl functions. We remark that the self-adjointness of $A_\eta$ for $\eta\not=\pm 2 c$ is also proven in
\cite{AMV14} using another approach.

%Using techniques associated to quasi boundary triples, we re-prove then 
%the self-adjointness of $A_\eta$,
%that was already shown in \cite{AMV14} via a different extension 
%theoretic approach.
%This construction only works for $\eta \neq \pm 2 
%c$; similar problems also
%appeared in \cite{AMV14, AMV15, AMV16}.
%For the interactions strengths $\eta = \pm 2 c$ we expect completely different physical
%properties like an infinite dimensional kernel or less smoothness of functions in its domain
%of definition.
%We assume that this issue can be solved by enlarging $T$, which shall be done in a
%subsequent work. Nevertheless, for the case $\eta \neq \pm 2 c$ we can 
%prove a Krein type resolvent formula
%for $A_\eta$, cf. Theorem \ref{theorem_Krein}, that will be the main 
%tool to show our results.

Let us describe the main results of this paper.
First, we discuss the spectral properties of the Dirac operator with an 
electrostatic $\delta$-shell interaction. Making use of some special properties of the Weyl function in the present situation 
the next result can be viewed as a consequence of the abstract resolvent formula and the corresponding Birman-Schwinger principle; for more details
and additional results see Theorem~\ref{theorem_Krein}.

\begin{thm} \label{theorem_1}
   Let $\eta \in \mathbb{R} \setminus \{ \pm 2 c \}$ and let $A_\eta$ be 
the Dirac operator with an electrostatic
   $\delta$-shell interaction of strength $\eta$. Then the essential 
spectrum is given by
   \begin{equation*}
     \sigma_{\mathrm{ess}}(A_\eta) = (-\infty, -m c^2] 
\cup [m c^2, \infty)
   \end{equation*}
   and the discrete spectrum in the gap $(-mc^2,mc^2)$ is finite, that is, 
   \begin{equation*}
    \sharp\bigl\{\sigma_{\rm d}(A_\eta)\cap(-mc^2,mc^2)\bigr\}<\infty.
   \end{equation*}
\end{thm}

The next result on the trace class property of 
the difference of the third powers of the resolvents 
of $A_\eta$ and
$A_0$ has important consequences for mathematical
scattering theory. In particular, it follows that the wave operators 
for the
scattering system $\{ A_\eta, A_0 \}$ exist and are complete and that the 
absolutely continuous parts of $A_\eta$ and
$A_0$ are unitarily equivalent. For more details see Theorem~\ref{theorem_resolvent_power_difference}, where also
a trace formula in terms of the Weyl function and its derivatives is provided.

\begin{thm} \label{theorem_2}
   Let $\eta \in \mathbb{R} \setminus \{ \pm 2 c \}$, let $A_\eta$ be 
the Dirac operator with an electrostatic
   $\delta$-shell interaction and let $\lambda \in \rho(A_\eta)\cap\rho(A_0)$.
   Then the operator
   \begin{equation*}
     (A_\eta - \lambda)^{-3} - (A_0 - \lambda)^{-3}
   \end{equation*}
   belongs to the trace class ideal.
\end{thm}

Our third and last main result in Theorem~\ref{theorem_nonrelativistic_limit} concerns the nonrelativistic limit of the Dirac 
operator with an electrostatic $\delta$-shell interaction.
We show that -- after subtracting the rest energy of the mass from the 
total energy -- $A_\eta$ converges in the norm 
resolvent sense to the Schr\"odinger operator
with an electric $\delta$-potential of strength~$\eta$ supported on 
$\Sigma$ times a projection onto the
upper components of the Dirac wave function, as~$c \rightarrow \infty$.
Hence, the Dirac operator with an electrostatic $\delta$-shell potential 
is the relativistic counterpart
of the Schr\"odinger operator with an electric $\delta$-interaction; cf. 
\cite[Chapter~6]{T92}. Since it is known that
the Schr\"odinger operator with a $\delta$-potential is a suitable 
idealized model for Schr\"odinger operators
with strongly localized regular potentials, cf. \cite{BEHL16},
the nonrelativistic limit yields a justification for the usage of 
$A_\eta$ as an idealized model
for the motion of a spin-$\frac{1}{2}$ particle in the presence of such 
a potential.
Furthermore, this theorem allows one to deduce spectral properties of $A_\eta$
for large $c$ from the well-known results on the Schr\"odinger operator with a $\delta$-potential.
Similar statements are already obtained for the one dimensional Dirac operator 
with $\delta$-interactions;
see \cite{AGHH05, CMP13, GS87}.
In a slightly simplified form
Theorem~\ref{theorem_nonrelativistic_limit} reads as follows.

\begin{thm} \label{theorem_3}
   Let $\eta \in \mathbb{R}$
   and let $A_\eta$ be the Dirac operator with an electrostatic
   $\delta$-shell interaction of strength $\eta$.
   Then, for any $\lambda \in \mathbb{C} \setminus \mathbb{R}$
   it holds
   \begin{equation*}
     \lim_{c \rightarrow \infty} \big( A_\eta - (\lambda + m c^2) \big)^{-1}
         = \left( -\frac{1}{2 m} \Delta + \eta \delta_\Sigma - \lambda 
\right)^{-1}
         \begin{pmatrix} I_2 & 0 \\ 0 & 0 \end{pmatrix},
   \end{equation*}
   where $I_2$ denotes the identity matrix in $\mathbb{C}^{2 \times 2}$
   and the convergence is in the operator norm.
\end{thm}

Finally, let us familiarize the reader with the structure of this paper. In 
Section~\ref{section_quasi_boundary_triples}
we provide a brief introduction to the general theory of quasi boundary triples and their Weyl functions.
The abstract results are formulated in the way they are needed to prove our main results.
Then, in Section~\ref{section_preliminaries} we introduce and 
investigate a quasi boundary triple which is suitable
to define and study the Dirac operator $A_\eta$ with an
electrostatic $\delta$-shell potential.
Using this quasi boundary triple we conclude the self-adjointness of $A_\eta$ and derive
a Krein type resolvent formula, which is an important tool in the proofs 
of our main results in Section~\ref{section_def_A_eta} and 
Section~\ref{section_nonrelativistic_limit}.
Finally, we have added the short Appendix~\ref{appa} on criteria for the boundedness of certain 
integral operators to ensure a self-contained presentation.

\subsection*{Notations}
The identity matrix in $\mathbb{C}^{n \times n}$ is denoted by
$I_n$. The Dirac 
matrices $\alpha_1, \alpha_2, \alpha_3$ and $\beta$ are 
\begin{equation} \label{def_Dirac_matrices}
   \alpha_j := \begin{pmatrix} 0 & \sigma_j \\ \sigma_j & 0 \end{pmatrix}
   \quad \text{and} \quad \beta := \begin{pmatrix} I_2 & 0 \\ 0 & -I_2 
\end{pmatrix},
\end{equation}
where $\sigma_j$, $j \in\{ 1, 2, 3 \}$, are the Pauli spin matrices
\begin{equation*}
   \sigma_1 := \begin{pmatrix} 0 & 1 \\ 1 & 0 \end{pmatrix}, \qquad
   \sigma_2 := \begin{pmatrix} 0 & -i \\ i & 0 \end{pmatrix}, \qquad
   \sigma_3 := \begin{pmatrix} 1 & 0 \\ 0 & -1 \end{pmatrix}.
\end{equation*}
Note that the Dirac matrices satisfy the anti-commutation relation
\begin{equation}\label{eq:commutation}
	\alpha_j\alpha_k + \alpha_k\alpha_j = 2\delta_{jk}I_4,\qquad
	j,k\in\{0,1,2,3\},
\end{equation}
with the convention $\alpha_0 := \beta$.

For vectors $x = (x_1, x_2, x_3)^{\top}$ we sometimes use the notation
$\alpha \cdot x := \sum_{j=1}^3 \alpha_j x_j$. 
Furthermore, $m$ and $c$ denote positive constants that stand for 
the mass of the particle and
the speed of light, respectively. The square root $\sqrt{\cdot}$ is fixed by $\sqrt{\lambda}\geq 0$
for $\lambda\geq 0$ and by $\mathrm{Im}\sqrt{\lambda}>0$ for $\lambda\in\mathbb C\setminus [0,\infty)$.

Throughout the text $\Sigma$ is the boundary of a bounded $C^\infty$-smooth domain 
in~$\mathbb R^3$ and $\sigma$ denotes 
the Hausdorff measure on $\Sigma$. We shall mostly work with the $L^2$-spaces 
$L^2(\mathbb{R}^3; \mathbb{C}^n)$ and $L^2(\Sigma; \mathbb{C}^n)$
of $\mathbb C^n$-valued square integrable functions, and more generally 
with $L^2(X; \mu; \mathbb{C}^n)$, where $(X,\mu)$ is a measure space.
We denote by
$C_c^\infty(\Omega; \mathbb{C}^n)$ 
the space of 
$\mathbb C^n$-valued smooth functions with 
compact support in an open set $\Omega \subset \mathbb{R}^3$, 
$H^k(\mathbb R^3; \mathbb{C}^n)$ stands for the usual Sobolev space of $k$-times 
weakly differentiable 
functions and $H^1_0(\mathbb R^3\setminus\Sigma; \mathbb{C}^n)$ is the closure of
$C_c^\infty(\mathbb R^3\setminus \Sigma; \mathbb{C}^n)$
with respect to the $H^1$-norm. In a similar manner, Sobolev spaces on $\Sigma$ are denoted by
$H^s(\Sigma; \mathbb{C}^n)$, $s\geq 0$.

For Hilbert spaces $X$ and $Y$ we denote by $\mathfrak B(X,Y)$ the space of all everywhere defined and 
bounded linear operators from $X$ to $Y$, in the case $X=Y$ 
we shall simply write $\mathfrak B(X)$.
We use $\mathfrak{S}_{p, \infty}(X, Y)$ for the weak 
Schatten--von Neumann ideal of order $p>0$.
Recall that a compact operator $K\colon X\rightarrow Y$ belongs to $\mathfrak{S}_{p, \infty}(X, Y)$,
if there exists a constant $\kappa$ such that the singular values $s_k(K)$ 
of $K$ satisfy $s_k(K) \leq \kappa k^{-1/p}$ for all $k \in \mathbb{N}$; cf. \cite{GK69} or \cite[Section 2.2]{BLL13_1}.
When no confusion can arise we will suppress the spaces $X$, $Y$ and simply write $\mathfrak{S}_{p, \infty}$. 
For a linear operator $T:X\rightarrow Y$ we denote the domain, range, and kernel by $\dom T$, $\ran T$, 
and $\ker T$, respectively. If $T$ is a closed operator in $X$ then its resolvent
set, spectrum, essential spectrum, discrete and point spectrum are denoted by 
$\rho(T)$, $\sigma(T)$, $\sigma_{\rm ess}(T)$, $\sigma_{\rm d}(T)$,
and $\sigma_{\rm p}(T)$, respectively.
Finally, $\sharp \sigma_{\rm{d}}(T)$ denotes the number of discrete eigenvalues counted with multiplicities.

\section{Quasi boundary triples and associated Weyl functions} \label{section_quasi_boundary_triples}

In this section we provide a brief introduction to boundary triple techniques in extension and 
spectral theory of symmetric and self-adjoint 
operators in Hilbert spaces. 
Here we present the necessary abstract material that is used in the formulation and proofs of our main results on 
Dirac operators with electrostatic $\delta$-shell interactions;
we refer the reader to \cite{BL07,BL12,BGP08,DHMS06,DM91,DM95,GG91} for more details, 
complete proofs and typical applications of boundary triples and their Weyl functions
in the theory of ordinary and partial differential operators.

In the following let $\mathfrak H$ be a Hilbert space with inner product $(\cdot,\cdot)_{\mathfrak H}$, let
$S$ be a densely defined closed symmetric operator in $\mathfrak H$, and let $S^*$ be the adjoint of $S$.

\begin{definition}\label{qbtdef}
Let $T$ be a linear operator in $\mathfrak H$ such that $\overline T=S^*$. 
A triple $\{\mathcal G,\Gamma_0,\Gamma_1\}$ is called a 
{\em quasi boundary triple} for $S^*$ if $(\mathcal G,(\cdot,\cdot)_{\mathcal G})$ 
is a Hilbert space and $\Gamma_0,\Gamma_1:\dom T\rightarrow\mathcal G$ are linear mappings such that
the following conditions {\rm (i)--(iii)} hold.
\begin{itemize}
 \item [{\rm (i)}] The abstract Green's identity
 \begin{equation*}
  (Tf,g)_{\mathfrak H}-(f,Tg)_{\mathfrak H}=(\Gamma_1 f,\Gamma_0 g)_{\mathcal G}-(\Gamma_0 f,\Gamma_1 g)_{\mathcal G}
 \end{equation*}
 is valid for all $f,g\in\dom T$.
 \item [{\rm (ii)}] The range of the mapping 
 $\Gamma=(\Gamma_0,\Gamma_1)^\top:\dom T\rightarrow \mathcal G\times\mathcal G$ is dense.
 \item [{\rm (iii)}] The operator $A_0:=T\upharpoonright\ker\Gamma_0$ is self-adjoint in $\mathfrak H$.
\end{itemize}
A quasi boundary triple is said to be a {\em generalized boundary triple} if $\ran\Gamma_0=\mathcal G$ and 
it is called an {\em ordinary boundary triple}
if $\ran \Gamma=\mathcal G\times\mathcal G$.
\end{definition}

The notion of quasi boundary triples was introduced in \cite{BL07} and further studied in \cite{BL12} and, 
e.g. \cite{BLL13,BLL13_1,BLL13_2}. 
It slightly extends the concepts of generalized boundary triples
from \cite{DM95} and ordinary boundary triples from \cite{B76,K75}. 
We note that the above definition of ordinary boundary triples is equivalent to the
usual definition in \cite{BGP08,DM91,GG91}; cf. \cite[Corollary 3.2]{BL07}. 
We also mention that a quasi boundary triple for $S^*$ exists if and only if $S$ admits self-adjoint extensions
in $\mathfrak H$, that is, if and only if the defect numbers $\dim\ker(S^*\pm i)$ coincide, 
and that the operator $T$ arising in Definition~\ref{qbtdef} is in general 
not unique (namely, when the defect numbers of $S$ are both infinite). 
Assume that $T\subset \overline T=S^*$ and let $\{\mathcal G,\Gamma_0,\Gamma_1\}$ 
be a quasi boundary triple for $S^*$.
Then according to \cite{BL07} one has
\begin{equation*}
 S=T\upharpoonright\bigl(\ker\Gamma_0\cap\ker\Gamma_1\bigr)
\end{equation*}
and the mapping $\Gamma=(\Gamma_0,\Gamma_1)^\top:\dom T\rightarrow\mathcal G\times\mathcal G$ is closable. 

Next we recall a variant of \cite[Theorem 2.3]{BL07} which in many situations is an efficient tool to verify that 
a certain boundary space $\mathcal G$ and  boundary mappings $\Gamma_0,\Gamma_1$  form a quasi boundary triple. 
We will make use of Theorem~\ref{theorem_guess} in the proof of Theorem~\ref{theorem_triple}.

\begin{thm} \label{theorem_guess}
Let $T$ be a linear operator in $\mathfrak H$, let $\mathcal G$ be a Hilbert space and assume that 
$\Gamma_0,\Gamma_1:\dom T\rightarrow\mathcal G$
are linear mappings which satisfy the following conditions~{\rm (i)--(iii)}.
    \begin{itemize}
    \item[(i)] The abstract Green's identity 
      \begin{equation*}
  (Tf,g)_{\mathfrak H}-(f,Tg)_{\mathfrak H}=(\Gamma_1 f,\Gamma_0 g)_{\mathcal G}-(\Gamma_0 f,\Gamma_1 g)_{\mathcal G}
      \end{equation*}
      holds for all $f,g \in \dom T$.
    \item[(ii)] The kernel and range of $\Gamma = ( \Gamma_0, \Gamma_1)^\top: \dom T \rightarrow \mathcal{G} \times \mathcal{G}$
      are dense in $\mathfrak H$ and $\mathcal G\times\mathcal G$, respectively.
    \item[(iii)] The restriction $T\upharpoonright\ker \Gamma_0$ contains a self-adjoint operator $A_0$.
  \end{itemize}
 Then 
 \begin{equation*}
  S:=T\upharpoonright\bigl(\ker\Gamma_0\cap\ker\Gamma_1\bigr)
 \end{equation*}
is a densely defined closed symmetric operator in $\mathfrak H$ and $\{\mathcal G,\Gamma_0,\Gamma_1\}$
is a quasi boundary triple for $\overline T=S^*$ such that
$A_0=T\upharpoonright\ker\Gamma_0$.
\end{thm}

In the following assume that $\{\mathcal G,\Gamma_0,\Gamma_1\}$ is a quasi boundary triple for 
$\overline T=S^*$ with $A_0=T\upharpoonright\ker\Gamma_0$.
The definition of the $\gamma$-field and Weyl function associated to the quasi boundary triple
$\{\mathcal G,\Gamma_0,\Gamma_1\}$ below is based 
on the direct sum decomposition
\begin{equation} \label{decomposition}
  \dom T = \dom A_0 \dot{+} \ker(T - \lambda)=\ker\Gamma_0\dot{+} \ker(T - \lambda),\qquad\lambda\in\rho(A_0).
\end{equation}
For ordinary and generalized boundary triples the $\gamma$-field and Weyl function were introduced in \cite{DM91} and \cite{DM95}. 
The definition for quasi boundary triples
is formally the same.

\begin{definition}\label{gammdef}
 The $\gamma$-field $\gamma$ and Weyl function $M$ corresponding to a quasi boundary triple  
 $\{\mathcal G,\Gamma_0,\Gamma_1\}$ for $\overline T=S^*$ are defined by
 \begin{equation*}
  \rho(A_0) \ni \lambda\mapsto\gamma(\lambda)=\bigl(\Gamma_0\upharpoonright\ker(T-\lambda)\bigr)^{-1},
 \end{equation*}
and
 \begin{equation*}
  \rho(A_0) \ni \lambda\mapsto M(\lambda)=\Gamma_1 \bigl(\Gamma_0\upharpoonright\ker(T-\lambda)\bigr)^{-1},
 \end{equation*}
respectively.
\end{definition}

It is immediate from the Definition~\ref{gammdef} and \eqref{decomposition} that $\gamma(\lambda)$, $\lambda\in\rho(A_0)$, 
is a linear operator defined on $\ran\Gamma_0$ which maps onto $\ker(T-\lambda)$.
Since $\ran\Gamma_0=\dom\gamma(\lambda)$ is dense in $\mathcal G$ by Definition~\ref{qbtdef}~(ii)  it is clear that 
$\gamma(\lambda)$, $\lambda\in\rho(A_0)$, is a densely defined operator from $\mathcal G$
into $\mathfrak H$. It can be shown with the help of the abstract Green's identity in Definition~\ref{qbtdef}~(i) that 
\begin{equation}\label{soso}
 \gamma(\lambda)^*=\Gamma_1 (A_0-\overline\lambda)^{-1}\in\mathfrak B(\mathfrak H,\mathcal G),\quad \lambda\in\rho(A_0),
\end{equation}
and this yields $\overline{\gamma(\lambda)}=\gamma(\lambda)^{**}\in\mathfrak B(\mathcal G,\mathfrak H)$ for 
$\lambda\in\rho(A_0)$; cf. 
\cite[Proposition 2.6]{BL07} or \cite[Proposition 6.13]{BL12}. Furthermore, for $\lambda,\mu\in\rho(A_0)$ 
and $\varphi\in\ran\Gamma_0$ one has
\begin{equation}\label{einmal}
 \gamma(\lambda)\varphi =\bigl(I+(\lambda-\mu)(A_0-\lambda)^{-1}\bigr)\gamma(\mu)\varphi.
\end{equation}
%and this identity extends by continuity to
%\begin{equation}\label{zweimal}
% \overline{\gamma(\lambda)} =\bigl(I+(\lambda-\mu)(A_0-\lambda)^{-1}\bigr)\overline{\gamma(\mu)}.
%\end{equation}
In particular, for all $\varphi\in\ran\Gamma_0$ the $\mathfrak H$-valued function $\lambda\mapsto\gamma(\lambda)\varphi$ 
is holomorphic on $\rho(A_0)$.
%and the $\mathfrak B(\mathcal G,\mathfrak H)$-valued function 
%$\overline{\gamma(\lambda)}$ is holomorphic on $\rho(A_0)$.
For $\lambda\in\rho(A_0)$ we shall later also make use of the relations 
\begin{equation}\label{jaha}
 \frac{d^k}{d\lambda^k} \gamma(\lambda) \varphi = k!(A_0-\lambda)^{-k} \gamma(\lambda) \varphi,\qquad k=0,1,\dots,
\end{equation}
for $\varphi \in \ran \Gamma_0$ and
\begin{equation}
 \frac{d^k}{d\lambda^k}\gamma\big(\overline \lambda\big)^*=k!\Gamma_1 (A_0-\lambda)^{-k-1},\qquad k=0,1,\dots,
\end{equation}
which were proved in \cite[Lemma~2.4]{BLL13_2}.
In the context of the $\gamma$-field we finally note that in the case of an ordinary or generalized boundary triple 
$\{\mathcal G,\Gamma_0,\Gamma_1\}$
the property $\ran\Gamma_0=\mathcal G$ implies $\gamma(\lambda)=\overline{\gamma(\lambda)}$. 
This leads to some obvious simplifications in the above considerations,
that is, \eqref{einmal} and \eqref{jaha} hold for all $\varphi \in \mathcal{G}$ and they can be viewed as 
equalities in $\mathfrak B(\mathcal G,\mathfrak H)$.

Next we collect some useful properties of the Weyl function $M$ associated to the quasi boundary triple 
$\{\mathcal G,\Gamma_0,\Gamma_1\}$.
Observe first that the values $M(\lambda)$, $\lambda\in\rho(A_0)$, are densely defined linear operators in 
$\mathcal G$ with $\dom M(\lambda)=\ran\Gamma_0$ and
$\ran M(\lambda)\subset \ran\Gamma_1$, and that
\begin{equation}
 M(\lambda)\Gamma_0 f_\lambda=\Gamma_1 f_\lambda,\quad f_\lambda\in\ker(T-\lambda).
\end{equation}
For $\lambda,\mu\in\rho(A_0)$ the Weyl function and $\gamma$-field are connected via the identity
\begin{equation}\label{holladi}
 M(\lambda)\varphi-M(\mu)^*\varphi=(\lambda-\overline\mu)\gamma(\mu)^*\gamma(\lambda)\varphi,\quad\varphi\in\ran\Gamma_0.
\end{equation}
This leads to $M(\lambda)\subset M\big(\overline\lambda\big)^*$, $\lambda\in\rho(A_0)$, and hence $M(\lambda)$ is a closable,
but in general unbounded operator in 
$\mathcal G$. Furthermore,
together with \eqref{einmal} one obtains from \eqref{holladi} that
\begin{equation}\label{juhu}
 M(\lambda)\varphi=M(\overline\mu)\varphi+(\lambda-\overline\mu)\gamma(\mu)^*\bigl(I+(\lambda-\mu)(A_0-\lambda)^{-1}\bigr)
 \gamma(\mu)\varphi,\quad\varphi\in\ran\Gamma_0,
\end{equation}
and hence for each $\varphi\in \ran \Gamma_0$ the $\mathcal G$-valued function $\lambda\mapsto M(\lambda)\varphi$ 
is holomorphic on $\rho(A_0)$. Moreover, due to \eqref{juhu}
the operator-valued function $\lambda\mapsto M(\lambda)$ can be viewed as the sum of a possibly unbounded operator 
$M(\overline\mu)$ and the function 
\begin{equation*}
 \lambda\mapsto (\lambda-\overline\mu)\gamma(\mu)^*\bigl(I+(\lambda-\mu)(A_0-\lambda)^{-1}\bigr)\gamma(\mu),
\end{equation*}
whose values are densely defined bounded operators. Thus it is clear that for $\lambda\in\rho(A_0)$ 
the derivatives of $M$ are bounded operators and from
\cite[Lemma 2.4]{BLL13_2} and \eqref{soso} one obtains for $\varphi \in \ran \Gamma_0$
 \begin{equation}\label{spitzenklasse}
 \frac{d^k}{d\lambda^k} M(\lambda) \varphi=k!\Gamma_1(A_0-\lambda)^{-k}\gamma(\lambda) \varphi,\qquad k=1,2,\dots.
\end{equation}
For $k=1$ and $\lambda\in\rho(A_0)\cap\mathbb R$ it follows directly from \eqref{holladi} that 
\begin{equation}\label{bittesehr}
 \frac{d}{d\lambda} (M(\lambda)\varphi,\varphi)_{\mathcal G}=\bigl(\gamma(\lambda)\varphi,
 \gamma(\lambda)\varphi\bigr)_{\mathfrak H}>0,
 \quad\varphi\in\ran\Gamma_0\setminus\{0\}.
\end{equation}
Similarly, as for the $\gamma$-field some of the above considerations simplify in the special case 
that $M$ is the Weyl function corresponding to an
ordinary or generalized boundary triple $\{\mathcal G,\Gamma_0,\Gamma_1\}$. 
Since in both situations $\ran\Gamma_0=\mathcal G$ it follows that the operators $M(\lambda)$ are defined on the
whole space $\mathcal G$ and hence \eqref{holladi} yields $M(\lambda)= M\big(\overline\lambda\big)^*$, 
so that $M(\lambda)\in\mathfrak B(\mathcal G)$ for all $\lambda\in\rho(A_0)$. Hence \eqref{spitzenklasse} holds 
for all $\varphi \in \mathcal{G}$ and
hence, as an equality in $\mathfrak{B}(\mathcal{G})$
and by \eqref{bittesehr} the $\mathfrak B(\mathcal G)$-valued operator function $M$ is monotonously non-decreasing 
on intervals in $\rho(A_0)\cap\mathbb R$.    

We shall use quasi boundary triples and their Weyl functions to describe self-adjoint extensions of $S$ and 
their spectral properties in Section~\ref{section_def_A_eta}.
For a linear operator $B$ in $\mathcal G$ we consider the extension
\begin{equation}\label{abba}
 A_{[B]}=T\upharpoonright \ker(\Gamma_0 + B \Gamma_1),
\end{equation}
that is, $f\in\dom T$ belongs to $\dom A_{[B]}$ if and only if $f$ satisfies the abstract boundary condition 
$\Gamma_0 f = -B \Gamma_1 f$. 
We emphasize that
the abstract boundary condition in \eqref{abba} is different to the usual choice $\ker(\Gamma_1-\Theta\Gamma_0)$, 
but is formally related to it via $\Theta=-B^{-1}$.
Note that for a
symmetric operator $B$ in $\mathcal G$ the abstract Green's identity yields
\begin{equation}\label{abab}
 (A_{[B]}f,g)_{\mathfrak H}-(f,A_{[B]}g)_{\mathfrak H}=
 -(\Gamma_1 f, B\Gamma_1 g)_{\mathcal G} + (B\Gamma_1 f,\Gamma_1 g)_{\mathcal G}=0,
\end{equation}
and hence the extension $A_{[B]}$ is symmetric in $\mathfrak H$. However, it is important to note that 
a self-adjoint operator $B$ does not automatically lead to a self-adjoint extension $A_{[B]}$. 
In fact, in contrast to the theory of ordinary boundary triples
in the more general situation of
quasi boundary triples and generalized boundary triples there is not a one-to-one correspondence
between self-adjoint parameters $B$ (or $\Theta$)
and self-adjoint extensions $A_{[B]}$ of the symmetric operator $S$ in $\mathfrak H$. 

The next theorem contains a variant of Krein's resolvent formula for canonical extensions which is useful to prove 
self-adjointness of such extensions;
cf. \cite[Theorem~2.8]{BL07}, \cite[Theorem 6.16]{BL12}, and \cite[Theorem 2.6]{BLL13_2}.

\begin{thm} \label{theorem_Krein_abstract}
Let $S$ be a densely defined closed symmetric operator in $\mathfrak H$ and let $\{\mathcal G,\Gamma_0,\Gamma_1\}$
be a quasi boundary triple for $\overline T=S^*$
with $A_0=T\upharpoonright\ker\Gamma_0$, $\gamma$-field $\gamma$ and Weyl function $M$.
Let $B$ be a linear operator in $\mathcal G$ and let $A_{[B]}$ 
be the extension of $S$ in \eqref{abba}. Then for all $\lambda\in\rho(A_0)$ one has
\begin{equation*}
 \ker(A_{[B]}-\lambda)=\bigl\{\gamma(\lambda)\varphi:\varphi\in\ker(I+B M(\lambda))\bigr\}
\end{equation*}
and, in particular, 
 $\lambda\in\sigma_{\mathrm{p}}(A_{[B]})$ if and only if $-1\in\sigma_{\mathrm{p}}(BM(\lambda))$. 
 Furthermore, if $\lambda\in\rho(A_0)$ is not an eigenvalue of $A_{[B]}$ then the following
 assertions {\rm (i)--(ii)} hold.
\begin{itemize}
 \item [{\rm (i)}] $g\in\ran(A_{[B]}-\lambda)$ if and only if $B\gamma(\overline\lambda)^*g\in\dom(I+BM(\lambda))^{-1}$; 
 \item [{\rm (ii)}] For all $ g\in\ran(A_{[B]}-\lambda)$ we have
 \begin{equation}\label{ressi}
  (A_{[B]}-\lambda)^{-1}g=(A_0-\lambda)^{-1}g-\gamma(\lambda)\bigl(I+BM(\lambda)\bigr)^{-1}B\gamma(\overline\lambda)^*g.
 \end{equation}
\end{itemize}
If $B\in\mathfrak B(\mathcal G)$ is self-adjoint and $(I+BM(\lambda_\pm))^{-1}\in\mathfrak B(\mathcal G)$ 
for some $\lambda_\pm\in\mathbb C^\pm$,
then $A_{[B]}$ is a self-adjoint operator in $\mathfrak H$ and 
\eqref{ressi} holds for all $\lambda\in\rho(A_0)\cap\rho(A_{[B]})$ and all $g\in\mathfrak H$.
\end{thm}

\section{Quasi boundary triples and Weyl functions for Dirac operators with singular interactions supported on $\Sigma$}
\label{section_preliminaries}

In this section we construct a quasi boundary triple which turns out to be suitable 
for the definition of Dirac operators with electrostatic $\delta$-shell interactions
supported on a compact $C^\infty$-surface $\Sigma$. % that split the Euclidean space in a smooth bounded domain 
%$\Omega$ and an unbounded domain $\mathbb R^3\setminus\overline\Omega$. 
We pay special attention to the properties 
of the associated Weyl function; these in turn will lead to a better understanding of the spectral properties
of Dirac operators with electrostatic $\delta$-shell interactions
in Section~\ref{section_def_A_eta}. For $\lambda\in (-mc^2,mc^2)$ the values $M(\lambda)$ 
of the Weyl function are closely related with the operators
$C_\sigma^\lambda$ in \cite{AMV14,AMV15,AMV16}; in this view the results  
on $M(\cdot)$ in Proposition~\ref{proposition_Weyl_function_mcsquare}~(ii) and 
Proposition~\ref{proposition_spectrum_Weyl_function} for $\lambda\in (-mc^2,mc^2)$ 
are known from \cite{AMV14,AMV15,AMV16}.

Recall first that the free Dirac operator
\begin{equation} \label{def_free_Dirac}
  A_0 f := -i c \sum_{j=1}^3 \alpha_j \partial_j f + m c^2 \beta f, \qquad \dom A_0 = H^1(\mathbb{R}^3; \mathbb{C}^4),
\end{equation}
where the Dirac matrices $\alpha_1, \alpha_2, \alpha_3$ and $\beta$ are given by \eqref{def_Dirac_matrices},
is self-adjoint in $L^2(\mathbb{R}^3; \mathbb{C}^4)$ and that 
\begin{equation} \label{spectrum_A_0}
  \sigma(A_0) = (-\infty, -m c^2] \cup [m c^2, \infty)
\end{equation}
holds; cf. \cite{T92} or \cite[Chapter~20]{W03}. 
Next, for $\lambda \in \rho(A_0)$ the resolvent of $A_0$ 
acts as
\begin{equation} \label{resolvent_A_0}
(A_0 - \lambda)^{-1} f(x) = \int_{\mathbb{R}^3} G_\lambda(x - y) f(y) \mathrm{d} y, 
\quad x \in \mathbb{R}^3,~f \in L^2(\mathbb{R}^3; \mathbb{C}^4),
\end{equation}
where the $\mathbb C^{4\times 4}$-valued integral kernel $G_\lambda$ is given by
\begin{equation} \label{def_G_lambda}
	G_\lambda(x) \!=\! \left( \frac{\lambda}{c^2} I_4 \! + m \beta 
	\! + \left( 1 \! - i \sqrt{\frac{\lambda^2}{c^2} - (m c)^2} |x| \right) \frac{i(\alpha \cdot x )}{c |x|^2} \right)
	\frac{e^{i \sqrt{\lambda^2/c^2\! - (m c)^2} |x|}}{4 \pi |x|};
\end{equation}
see \cite[Section~1.E]{T92} or \cite[Lemma~2.1]{AMV15}.
The explicit form of this integral kernel will be particularly important in our
further considerations. 
Moreover, if we denote by $-\Delta$ the self-adjoint Laplacian in $L^2(\mathbb{R}^3; \mathbb{C})$ 
defined on $H^2(\mathbb{R}^3; \mathbb{C})$, then using~\eqref{eq:commutation} we get
\begin{equation} \label{A_0_square}
  A_0^2  = (-c^2 \Delta + m^2 c^4) I_4, \qquad \dom A_0^2 = H^2(\mathbb{R}^3; \mathbb{C}^4);
\end{equation}
cf. \cite[Korollar~20.2]{W03} (here, the case $m=c=1$ is considered, which is up to a scaling transform
equivalent to our case). 
The operator $(-c^2 \Delta + m^2 c^4) I_4$ is understood as a $4 \times 4$ block operator with diagonal structure, 
where each diagonal entry acts as $-c^2 \Delta + m^2 c^4$.

In the following let $\Sigma$ be the boundary of a bounded $C^\infty$-domain in $\mathbb R^3$.
For the definition of the quasi boundary triple in Theorem~\ref{theorem_triple} below we first introduce 
two integral operators associated with the function
\begin{equation*} 
  G_0(x) = \frac{e^{-m c |x|}}{4 \pi |x|} 
             \left(m \beta + \big( 1 + m c |x| \big) \frac{i(\alpha \cdot x )}{c |x|^2} \right).
\end{equation*}
Note that there exist constants $\kappa, R > 0$ such that
\begin{equation} \label{asymptotics_G_0}
  |G_0(x)| \leq \kappa \begin{cases} |x|^{-2},&~|x | < R, \\
                                       e^{-m c |x|},&~ |x| \geq R. \end{cases}
\end{equation}
Now define the strongly singular integral 
operator~$M: L^2(\Sigma; \mathbb{C}^4) \rightarrow L^2(\Sigma; \mathbb{C}^4)$ by
\begin{equation} \label{def_M}
    M \varphi(x) := 
    \lim_{\varepsilon \searrow 0} \int_{|x - y| > \varepsilon} G_{0}(x - y) \varphi(y) \mathrm{d} \sigma(y),
    \quad x \in \Sigma,~ \varphi \in L^2(\Sigma; \mathbb{C}^4).
\end{equation}
It was shown in \cite[Lemma~3.3 and Lemma~3.7]{AMV14} 
that $M$ is a bounded self-adjoint operator (to see this, note that $c M = C_\sigma$ in the 
notation of \cite[Lemma~3.3]{AMV14}, where $m$ in \cite[Lemma~3.1]{AMV14} is replaced by $mc$).
Furthermore, we define the mapping $\gamma: L^2(\Sigma; \mathbb{C}^4) \rightarrow L^2(\mathbb{R}^3; \mathbb{C}^4)$ by
\begin{equation} \label{def_gamma}
  \gamma \varphi(x) := \int_\Sigma G_0(x - y) \varphi(y) \mathrm{d} \sigma(y), 
  \qquad x \in \mathbb{R}^3, ~\varphi \in L^2(\Sigma; \mathbb{C}^4),
\end{equation}
and observe that \eqref{asymptotics_G_0} and Proposition~\ref{proposition_integral_operators2} 
imply that $\gamma$ is bounded and everywhere defined. 

The following auxiliary result ensures that the operator $T$ in \eqref{def_T} below is well-defined.

\begin{lem} \label{lemma_decomposition}
  Let $f, g \in H^1(\mathbb{R}^3; \mathbb{C}^4)$ and $\varphi, \psi \in L^2(\Sigma; \mathbb{C}^4)$
  such that $f + \gamma \varphi = g + \gamma \psi$. Then $f=g$ and $\varphi = \psi$.
\end{lem} 
\begin{proof}
  From $f + \gamma \varphi = g + \gamma \psi$ it follows 
  $\gamma(\psi - \varphi) = f - g \in H^1(\mathbb{R}^3; \mathbb{C}^4) = \dom A_0$.
  Let $h \in H^1_0(\mathbb{R}^3 \setminus \Sigma; \mathbb{C}^4)$. Then the self-adjointness 
  of $A_0$ and \cite[Lemma~2.10]{AMV14} yield
  \begin{equation*}
    \begin{split}
      (A_0 \gamma(\psi - \varphi), h)_{L^2(\mathbb{R}^3; \mathbb{C}^4)}
        &= (\gamma(\psi - \varphi), A_0 h)_{L^2(\mathbb{R}^3; \mathbb{C}^4)} \\
      &= \big(\psi - \varphi, (A_0^{-1} A_0 h)|_\Sigma \big)_{L^2(\Sigma; \mathbb{C}^4)} = 0
    \end{split}
  \end{equation*}
  and, since $H^1_0(\mathbb{R}^3 \setminus \Sigma; \mathbb{C}^4)$ is dense in $L^2(\mathbb{R}^3; \mathbb{C}^4)$,
  we conclude $A_0 \gamma(\psi - \varphi) = 0$. Now $0 \in \rho(A_0)$ yields $f - g = \gamma(\psi - \varphi) = 0$.
  
  It remains to show $\varphi = \psi$. For $k \in L^2(\mathbb{R}^3; \mathbb{C}^4)$ and \cite[Lemma~2.10]{AMV14} we obtain
  \begin{equation*}
    0 = (\gamma(\psi - \varphi), k)_{L^2(\mathbb{R}^3; \mathbb{C}^4)}
      = \big(\psi - \varphi, (A_0^{-1} k)|_\Sigma \big)_{L^2(\Sigma; \mathbb{C}^4)},
  \end{equation*}
  and since the range of the mapping $L^2(\mathbb{R}^3; \mathbb{C}^4) \ni k \mapsto (A_0^{-1} k)|_\Sigma$ is $H^{1/2}(\Sigma; \mathbb{C}^4)$
  we conclude $\varphi = \psi$. 
\end{proof}

Now, we define the operator $T$ in $L^2(\mathbb{R}^3; \mathbb{C}^4)$ via
\begin{equation} \label{def_T}
  \begin{split}
   T (f + \gamma \varphi) &:= A_0 f,\\
    \dom T &:= \big\{ f + \gamma \varphi: f \in H^1(\mathbb{R}^3; \mathbb{C }^4), \varphi \in L^2(\Sigma; \mathbb{C}^4) \big\}.
  \end{split}
\end{equation}
In the following elements in $\dom T$ 
will always be written in the form $f + \gamma \varphi$ with $f \in H^1(\mathbb{R}^3; \mathbb{C}^4)$
and $\varphi \in L^2(\Sigma; \mathbb{C}^4)$;
this decomposition is unique because of Lemma~\ref{lemma_decomposition} and hence $T$ is well-defined.

\begin{thm} \label{theorem_triple}
  Let $T$ be given by \eqref{def_T}. Then, the operator
  \begin{equation}\label{sss}
   S:=A_0\upharpoonright H^1_0(\mathbb{R}^3 \setminus \Sigma; \mathbb{C}^4)
  \end{equation}
is densely defined, closed and symmetric in $L^2(\mathbb{R}^3; \mathbb{C}^4)$ 
and $\{L^2(\Sigma; \mathbb{C}^4), \Gamma_0, \Gamma_1\}$,
where
\begin{equation}\label{def_triple}
 \Gamma_0(f + \gamma \varphi) = \varphi\quad\text{and}\quad \Gamma_1(f + \gamma \varphi) 
 = f|_\Sigma + M \varphi,\qquad f + \gamma \varphi\in\dom T, 
\end{equation}
is a quasi boundary triple for $\overline T=S^*$ such that $T\upharpoonright\ker\Gamma_0$
coincides with the free Dirac operator $A_0$ in \eqref{def_free_Dirac}.
\end{thm}

\begin{proof}
 We shall use Theorem~\ref{theorem_guess} to prove the claim. Note that the mappings
 $\Gamma_0$ and $\Gamma_1$ are well-defined by Lemma~\ref{lemma_decomposition}.
 First, we check Green's identity in Theorem~\ref{theorem_guess}~(i).
 For $f + \gamma \varphi, g + \gamma \psi \in \dom T$ it follows from \eqref{def_T} and the 
  self-adjointness of $A_0$ that
  \begin{equation*}
    \begin{split}
      \big( T (f + \gamma \varphi)&, g + \gamma \psi \big)_{L^2(\mathbb{R}^3; \mathbb{C}^4)}
            - \big( f + \gamma \varphi, T( g + \gamma \psi) \big)_{L^2(\mathbb{R}^3; \mathbb{C}^4)} \\
       &= \big( A_0 f, g + \gamma \psi \big)_{L^2(\mathbb{R}^3; \mathbb{C}^4)}
            - \big( f + \gamma \varphi, A_0 g \big)_{L^2(\mathbb{R}^3; \mathbb{C}^4)} \\
       &= \big( A_0 f, \gamma \psi \big)_{L^2(\mathbb{R}^3; \mathbb{C}^4)}
            - \big( \gamma \varphi, A_0 g \big)_{L^2(\mathbb{R}^3; \mathbb{C}^4)}.
    \end{split}
  \end{equation*}
  Since
  \begin{equation*}
    \begin{split}
      \big( A_0 f, \gamma \psi \big)_{L^2(\mathbb{R}^3; \mathbb{C}^4)} = 
          \big( f|_\Sigma, \psi \big)_{L^2(\Sigma; \mathbb{C}^4)} \quad \text{and} \quad
      \big( \gamma \varphi, A_0 g \big)_{L^2(\mathbb{R}^3; \mathbb{C}^4)} = 
          \big( \varphi, g|_\Sigma \big)_{L^2(\Sigma; \mathbb{C}^4)}
    \end{split}
  \end{equation*}
  by \cite[Lemma 2.10]{AMV14} and $M$ is self-adjoint we obtain
  \begin{equation*}
    \begin{split}
      \big( T &(f + \gamma \varphi), g + \gamma \psi \big)_{L^2(\mathbb{R}^3; \mathbb{C}^4)}
            - \big( f + \gamma \varphi, T( g + \gamma \psi) \big)_{L^2(\mathbb{R}^3; \mathbb{C}^4)} \\
       &= \big( f|_\Sigma, \psi \big)_{L^2(\Sigma; \mathbb{C}^4)}  
            - \big( \varphi, g|_\Sigma \big)_{L^2(\Sigma; \mathbb{C}^4)} \\
       &= \big( f|_\Sigma + M\varphi, \psi \big)_{L^2(\Sigma; \mathbb{C}^4)}
            - \big( \varphi, g|_\Sigma + M\psi\big)_{L^2(\Sigma; \mathbb{C}^4)} \\
       &= \big( \Gamma_1 (f + \gamma \varphi), \Gamma_0 (g + \gamma \psi) \big)_{L^2(\Sigma; \mathbb{C}^4)}
            - \big( \Gamma_0 (f + \gamma \varphi), \Gamma_1 (g + \gamma \psi) \big)_{L^2(\Sigma; \mathbb{C}^4)},
    \end{split}
  \end{equation*}
  that is, assumption (i) in Theorem \ref{theorem_guess} holds. 

  Next, we prove that $\Gamma$ has dense range. To show this
  consider $(\psi, \xi) \in (\ran \Gamma)^\bot$.
  Then, we have 
  \begin{equation}\label{zu}
    \big(\psi, \Gamma_0 (f + \gamma \varphi) \big)_{L^2(\Sigma; \mathbb{C}^4)} 
      + \big(\xi, \Gamma_1 (f + \gamma \varphi) \big)_{L^2(\Sigma; \mathbb{C}^4)} = 0
  \end{equation}
  for all $f + \gamma \varphi \in \dom T$. The special choice
  $\varphi = 0$ leads to 
  \begin{equation*}
    0 = \big(\xi, \Gamma_1 f \big)_{L^2(\Sigma; \mathbb{C}^4)} = \big(\xi, f|_\Sigma \big)_{L^2(\Sigma; \mathbb{C}^4)},\quad 
    f \in H^1(\mathbb{R}^3; \mathbb{C}^4).
  \end{equation*}
  Since the trace operator has dense range we conclude 
  $\xi = 0$ and therefore \eqref{zu} reduces to
  \begin{equation*}
    0 = \big(\psi, \Gamma_0 (f + \gamma \varphi) \big)_{L^2(\Sigma; \mathbb{C}^4)} = (\psi, \varphi)_{L^2(\Sigma; \mathbb{C}^4)} 
  \end{equation*}
  for all $\varphi \in L^2(\Sigma; \mathbb{C}^4)$. Thus $\psi = 0$ and it follows that $\ran \Gamma$ is dense. It is clear that
  $\ker\Gamma=H^1_0(\mathbb{R}^3 \setminus \Sigma; \mathbb{C}^4)$ is dense in $L^2(\mathbb{R}^3; \mathbb{C}^4)$.
  We have shown that assumption~(ii) in Theorem~\ref{theorem_guess} is satisfied.
  Finally, assumption (iii) in Theorem \ref{theorem_guess} holds, since 
  $T\upharpoonright\ker \Gamma_0$ is the free Dirac operator.

  Now it follows from Theorem~\ref{theorem_guess} that
  $\{L^2(\Sigma;\mathbb C^4), \Gamma_0, \Gamma_1 \}$ is a quasi boundary triple for $\overline T=S^*$, 
  where $S$ is the restriction of $T$ onto 
  $\ker\Gamma=H^1_0(\mathbb{R}^3 \setminus \Sigma; \mathbb{C}^4)$ 
  in~\eqref{sss}.
\end{proof}

\begin{remark}
  Note that $\ran \Gamma_0 = L^2(\Sigma; \mathbb{C}^4)$ in Theorem~\ref{theorem_triple} and hence the triple
  $\{L^2(\Sigma;\mathbb C^4), \Gamma_0, \Gamma_1\}$ is even a generalized 
  boundary triple in the sense of \cite{DM95}; cf. Definition~\ref{qbtdef}.
  In particular, it follows that the values of the corresponding $\gamma$-field and Weyl function 
  (see Proposition~\ref{proposition_gamma_Weyl}) are everywhere defined and bounded operators. 
  In the case that $\gamma$ in \eqref{def_gamma} and 
  the strongly singular integral operator $M$ in \eqref{def_M}
  are only considered on a dense subspace of $L^2(\Sigma;\mathbb C^4)$ 
  the corresponding triple in Theorem~\ref{theorem_triple} is still a quasi boundary triple. 
\end{remark}

Next we compute the $\gamma$-field and the Weyl function associated to the quasi (or generalized) boundary triple 
$\{L^2(\Sigma;\mathbb C^4), \Gamma_0, \Gamma_1\}$. It turns out that the operators $\gamma$ and $M$ introduced
in \eqref{def_gamma} and \eqref{def_M}, respectively,
coincide with the values of the $\gamma$-field and the Weyl function at the
point $\lambda=0$.

\begin{prop} \label{proposition_gamma_Weyl}
  Let $\{L^2(\Sigma;\mathbb C^4), \Gamma_0, \Gamma_1\}$ be as in Theorem~\ref{theorem_triple} 
  and let $G_\lambda$ be the integral kernel of the resolvent
  of the free Dirac operator $A_0$ in \eqref{def_G_lambda}. Then the following holds.
  \begin{itemize}
    \item[\rm (i)] The $\gamma$-field is holomorphic on 
    $\rho(A_0)= \mathbb{C} \setminus ( (-\infty, -m c^2] \cup [m c^2, \infty) )$, the operators 
    $\gamma(\lambda): L^2(\Sigma; \mathbb{C}^4) \rightarrow L^2(\mathbb{R}^3; \mathbb{C}^4)$
    are everywhere defined and bounded, and  given by 
    \begin{equation*} 
      \begin{split}
        \gamma(\lambda) \varphi(x) = \int_\Sigma G_\lambda(x - y) \varphi(y) \mathrm{d} \sigma(y),
        \quad x \in \mathbb{R}^3, ~\varphi \in L^2(\Sigma; \mathbb{C}^4).
      \end{split}
    \end{equation*}
    Their adjoints  $\gamma(\lambda)^*: L^2(\mathbb{R}^3; \mathbb{C}^4) \rightarrow L^2(\Sigma; \mathbb{C}^4)$ are
    \begin{equation*}
      \gamma(\lambda)^* f(x) = \int_{\mathbb{R}^3} G_{\overline{\lambda}}(x - y) f(y) \mathrm{d} y,
      \quad x \in \Sigma, ~f \in L^2(\mathbb{R}^3; \mathbb{C}^4).
    \end{equation*}
    The operators $\gamma(\lambda)$ and $\gamma(\lambda)^*$ are compact for all $\lambda \in \rho(A_0)$.
    
    \item[\rm (ii)] The Weyl function $M(\cdot)$ is holomorphic on 
    $\rho(A_0)= \mathbb{C} \setminus ( (-\infty, -m c^2] \cup [m c^2, \infty) )$, the operators 
    $M(\lambda): L^2(\Sigma; \mathbb{C}^4) \rightarrow L^2(\Sigma; \mathbb{C}^4)$ are everywhere defined and bounded, and 
    given by
    \begin{equation*}
      \begin{split}
        &M(\lambda) \varphi(x) := 
        \lim_{\varepsilon \searrow 0} \int_{|x - y| > \varepsilon} G_{\lambda}(x - y) \varphi(y) \mathrm{d} \sigma(y),
        \quad x \in \Sigma, ~\varphi \in L^2(\Sigma; \mathbb{C}^4).
      \end{split}
    \end{equation*}
    %The operator $M(\lambda)$ is bounded and everywhere defined for any $\lambda \in \rho(A_0)$.     
  \end{itemize}
\end{prop}
\begin{proof}
  (i) By \eqref{soso} we have
  $\gamma(\lambda)^* = \Gamma_1 (A_0 - \overline{\lambda})^{-1}$, $\lambda\in\rho(A_0)$, and 
  this operator has the explicit representation
  \begin{equation}\label{gh}
    \gamma(\lambda)^* f(x) = \int_{\mathbb{R}^3} G_{\overline{\lambda}}(x - y) f(y) \mathrm{d} y,
    \quad x \in \Sigma,~f \in L^2(\mathbb{R}^3; \mathbb{C}^4).
  \end{equation} 
  From the properties of the trace map we conclude $\ran \gamma(\lambda)^* = H^{1/2}(\Sigma; \mathbb{C}^4)$,
  which together with the closed graph theorem implies that $\gamma(\lambda)^*$ 
  is bounded and everywhere defined as an operator from $L^2(\mathbb{R}^3; \mathbb{C}^4)$ onto $H^{1/2}(\Sigma; \mathbb{C}^4)$.
  Since the embedding 
  $H^{1/2}(\Sigma; \mathbb{C}^4) \hookrightarrow L^2(\Sigma; \mathbb{C}^4)$ is compact it follows
  that $\gamma(\lambda)^*$, $\lambda\in\rho(A_0)$, is compact.
  
  Next, we analyze $\gamma(\lambda)$, $\lambda\in\rho(A_0)$. As $\Gamma_0$ is surjective it follows that 
  $\gamma(\lambda)$ is everywhere defined and bounded 
  (see Section~\ref{section_quasi_boundary_triples})
  and since $\gamma(\lambda)^*$ is compact also $\gamma(\lambda) = \gamma(\lambda)^{**}$ is compact. 
  Moreover, using \eqref{gh} and $\overline{G_{\overline\lambda}(x-y)}=G_\lambda(x-y)$ we obtain
  \begin{equation*}
    \begin{split}
      \big( \gamma(\lambda) \varphi, f \big)_{L^2(\mathbb{R}^3; \mathbb{C}^4)} 
          &= \big( \varphi, \gamma(\lambda)^* f \big)_{L^2(\Sigma; \mathbb{C}^4)} 
          = \int_\Sigma \varphi(x) 
               \overline{\int_{\mathbb{R}^3} G_{\overline{\lambda}}(x - y) f(y) \mathrm{d} y} \mathrm{d} \sigma(x) \\
      &= \int_{\mathbb{R}^3} \int_\Sigma G_\lambda(x - y) \varphi(x) \mathrm{d} \sigma(x) \overline{f(y)} \mathrm{d} y
    \end{split}
  \end{equation*}
  for all $\varphi \in L^2(\Sigma; \mathbb{C}^4)$ and 
  $f \in L^2(\mathbb{R}^3; \mathbb{C}^4)$, which yields the integral representation of $\gamma(\lambda)$.

  (ii) In order to compute $M(\lambda)$, $\lambda\in\rho(A_0)$, we use  
  $\gamma(\lambda) = ( I_4 + \lambda (A_0 - \lambda)^{-1}) \gamma(0)$ and $\gamma(0)=\gamma$; 
  cf. \eqref{einmal} and \eqref{def_gamma}. 
  It follows from the definition of $\Gamma_1$ in \eqref{def_triple} that   
  \begin{equation}\label{eq:decomposition_M}
    M(\lambda) \varphi = \Gamma_1 \gamma(\lambda) \varphi = \Gamma_1 \big( I_4 + \lambda (A_0 - \lambda)^{-1} \big) \gamma \varphi
        = M \varphi + \bigl(\lambda (A_0 - \lambda)^{-1} \gamma \varphi\bigr)|_\Sigma.
  \end{equation}
  We shall derive an integral formula for $(\lambda (A_0 - \lambda)^{-1} \gamma \varphi)|_\Sigma$ next.
  First note that for all $g \in C_c^\infty(\mathbb{R}^3; \mathbb{C}^4)$ and almost all $x \in \mathbb{R}^3$ we have
  \begin{equation*}
    \begin{split}
      \int_{\mathbb{R}^3} \big[G_\lambda(x - y) - G_0(x - y) \big] g(y) \mathrm{d} y 
          &= (A_0 - \lambda)^{-1} g(x) - A_0^{-1} g(x) \\
      &= \lambda (A_0 - \lambda)^{-1} A_0^{-1} g(x) \\
      &= \lambda \int_{\mathbb{R}^3} G_\lambda(x - y) \int_{\mathbb{R}^3} G_0(y - z) g(z) \mathrm{d} z \mathrm{d} y \\
      &= \lambda \int_{\mathbb{R}^3} g(z) \int_{\mathbb{R}^3} G_\lambda(x-y) G_0(y-z) \mathrm{d} y \mathrm{d} z \\
      &= \lambda \int_{\mathbb{R}^3} g(y) \int_{\mathbb{R}^3} G_\lambda(x-z) G_0(y-z) \mathrm{d} z \mathrm{d} y,
    \end{split}
  \end{equation*}
  where Fubini's theorem, a permutation of the variables $y$ and $z$ and the identity 
  $G_0(y - z) = G_0(z - y)$ were used in the last two steps. Hence, 
  \begin{equation*}
    G_\lambda(x - y) - G_0(x - y) = \lambda \int_{\mathbb{R}^3} G_\lambda(x-z) G_0(y-z) \mathrm{d} z
  \end{equation*}
  is true for almost all $x, y \in \mathbb{R}^3$. This can be extended by the continuity of $G_\lambda$ for all $x \neq y$.
  Employing again Fubini's theorem, we deduce for $x \in \Sigma$
  \begin{equation}\label{nimmdas}
    \begin{split}
      \lambda (A_0 - \lambda)^{-1} \gamma \varphi(x) 
          &= \lambda \int_{\mathbb{R}^3} G_\lambda(x - y) \int_{\Sigma} G_0(y-z) \varphi(z) \mathrm{d} \sigma(z) \mathrm{d} y \\
      &= \lambda \int_{\Sigma} \varphi(z) \int_{\mathbb{R}^3} G_\lambda(x - y) G_0(y-z) \mathrm{d} y \mathrm{d} \sigma(z) \\
      &= \int_{\Sigma} \varphi(z) \big[ G_\lambda(x - z) - G_0(x - z) \big] \mathrm{d} \sigma(z).
    \end{split}
  \end{equation}
  Since $\lambda(A_0 - \lambda)^{-1} \gamma \varphi|_\Sigma \in L^2(\Sigma; \mathbb{C}^4)$ the last term 
  is finite for almost all $x \in \Sigma$. Therefore, 
  $\varphi \big[ G_\lambda(x - \cdot) - G_0(x - \cdot) \big] \in L^1(\Sigma; \mathbb{C}^4)$
  for almost all $x \in \Sigma$. 
  Hence, for these $x$ we obtain from \eqref{eq:decomposition_M}, \eqref{def_M}, \eqref{nimmdas}, and dominated convergence that
  \begin{equation*}
    \begin{split}
      M(\lambda) \varphi(x) &= M \varphi (x) + \lambda (A_0 - \lambda)^{-1} \gamma \varphi(x) \\
      &= \lim_{\varepsilon \searrow 0} \int_{|x - y| > \varepsilon} G_{0}(x - y) \varphi(y) \mathrm{d} \sigma(y)\\
          &\qquad + \lim_{\varepsilon \searrow 0} \int_{|x-y|>\varepsilon} \big[ G_\lambda(x - y) - G_0(x - y) \big] 
          \varphi(y) \mathrm{d} \sigma(y) \\
      &= \lim_{\varepsilon \searrow 0} \int_{|x - y| > \varepsilon} G_{\lambda}(x - y) \varphi(y) \mathrm{d} \sigma(y);
    \end{split}
  \end{equation*}
  this shows the representation of $M(\lambda)$ in (ii). Note that the operators $M(\lambda)$, $\lambda\in\rho(A_0)$, 
  are everywhere defined and bounded since $\ran\Gamma_0=L^2(\Sigma;\mathbb C^4)$
  (see Section~\ref{section_quasi_boundary_triples}).
\end{proof}

In order to show self-adjointness and to discuss the spectral properties of Dirac operators with $\delta$-shell interactions
in the next section some more information on the Weyl function $M(\cdot)$ is necessary. 
Most of the results in the next two propositions are already contained in
\cite{AMV15, AMV16} in a similar form; for the convenience of the reader 
we collect and trace them back to those in \cite{AMV15, AMV16}.
First, we show that the Weyl function $M(\cdot)$ admits an extension  
to the points~$\lambda = \pm m c^2$; this extension is in accordance with the integral representation of 
$M(\cdot)$ in Proposition~\ref{proposition_gamma_Weyl}~(ii)
in the sense that the functions $G_{\pm m c^2}$ in Proposition~\ref{proposition_Weyl_function_mcsquare}~(i) 
below coincide with the 
  the integral kernel $G_\lambda$ of the resolvent
  of the free Dirac operator $A_0$ in \eqref{def_G_lambda} at $\lambda=\pm mc^2$. 
  The assertion in Proposition~\ref{proposition_Weyl_function_mcsquare}~(ii) is a variant of \cite[Lemma 3.2]{AMV15}.

\begin{prop} \label{proposition_Weyl_function_mcsquare}
  Let $\{L^2(\Sigma;\mathbb C^4), \Gamma_0, \Gamma_1 \}$ be the quasi boundary triple in Theorem~\ref{theorem_triple}
  with corresponding Weyl function $M(\cdot)$.
  Then the following assertions hold. 
  \begin{itemize}
    \item[\rm (i)] The limits 
    \begin{equation*}
      M(m c^2) := \lim_{\lambda \nearrow  m c^2} M(\lambda)
      \quad\text{and}\quad M(-m c^2) := \lim_{\lambda \searrow -m c^2} M(\lambda)
    \end{equation*}
    exist in the operator norm on $\mathfrak{B}\big(L^2(\Sigma; \mathbb{C}^4)\big)$. The corresponding limit operators 
    $M(\pm m c^2): L^2(\Sigma; \mathbb{C}^4) \rightarrow L^2(\Sigma; \mathbb{C}^4)$ are given by 
    \begin{equation*}
      \begin{split}
        M(\pm m c^2) \varphi(x) =  \lim_{\varepsilon \searrow 0} \int_{|x - y| > \varepsilon}
            G_{\pm m c^2}(x - y) \varphi(y) \mathrm{d} \sigma(y)&,\\
            \quad x \in \Sigma, ~\varphi \in L^2(\Sigma; \mathbb{C}^4)&,
      \end{split}
    \end{equation*}
    where the functions $G_{\pm m c^2}$ are defined by
    \begin{equation*}
      G_{\pm m c^2}(x) = \left( m (\beta \pm I_4) + \frac{i(\alpha\cdot x)}{c |x|^2}\right)
                     \frac{1}{4 \pi |x|}.
    \end{equation*}
    
    \item[\rm (ii)] The Weyl function $\lambda\mapsto M(\lambda)$ is uniformly bounded on $[-mc^2,mc^2]$, i.e.
    $$M_0 := \sup_{\lambda \in [-m c^2, m c^2]} \| M(\lambda) \|<\infty.$$
  \end{itemize} 
\end{prop}

\begin{proof}
  (i) We discuss only the case $\lambda \nearrow m c^2$, the statement for $\lambda \searrow -m c^2$
  can be proved in exactly the same way. We define the singular integral operator 
  \begin{equation*}
    \begin{split}
      &C \varphi(x) =  \lim_{\varepsilon \searrow 0} \int_{|x - y| > \varepsilon}
          G_{m c^2}(x - y) \varphi(y) \mathrm{d} \sigma(y),
          \quad x \in \Sigma, ~\varphi \in L^2(\Sigma; \mathbb{C}^4),
    \end{split}
  \end{equation*}
  and show that $M(\lambda)$ converges to $C$ in the operator norm as $\lambda \nearrow m c^2$. Note that for 
  $\lambda \in (-m c^2, m c^2)$ we have 
  \begin{equation*}
    \begin{split}
      C - M(\lambda) = T_1(\lambda) + T_2(\lambda) + T_3(\lambda),
    \end{split}
  \end{equation*}
  where for $j \in \{ 1, 2, 3 \}$ the operator $T_j(\lambda): L^2(\Sigma; \mathbb{C}^4) \rightarrow L^2(\Sigma; \mathbb{C}^4)$ 
  is of the form
  \begin{equation*}
    T_j(\lambda) \varphi(x) := \lim_{\varepsilon \searrow 0} \int_{|x-y| > \varepsilon} t_j^\lambda(x-y) \varphi(y) \mathrm{d} \sigma(y), 
    \quad x \in \Sigma,~\varphi \in L^2(\Sigma; \mathbb{C}^4),
  \end{equation*}
  with
  \begin{equation*}
    \begin{split}
      t_1^\lambda(x) &:= \left( m - \frac{\lambda}{c^2} \right) \frac{ e^{-\sqrt{(m c)^2-\lambda^2/c^2} |x|}}{4 \pi |x|} I_4; \\
      t_2^\lambda(x) &:= -\sqrt{(m c)^2 - \frac{\lambda^2}{c^2}}\, \frac{i(\alpha\cdot x)}{c |x|}\,
          \frac{ e^{-\sqrt{(m c)^2-\lambda^2/c^2} |x|}}{4 \pi |x|}; \\ 
      t_3^\lambda(x) &:= \left( \frac{i(\alpha \cdot x)}{c |x|^2}   + m(I_4 + \beta) \right)
          \frac{1 - e^{-\sqrt{(m c)^2-\lambda^2/c^2} |x|}}{4 \pi |x|}.
   \end{split}
  \end{equation*}
  We will see that the operators $T_1(\lambda), T_2(\lambda)$ and $T_3(\lambda)$ are bounded and everywhere defined,
  which yields then that also $C$ has this property.
  
  First, since 
  $|t_1^\lambda(x)| \leq \left( m - \lambda/c^2 \right) (4 \pi |x|)^{-1}$ for
  $x\in \mathbb{R}^3$, Proposition~\ref{proposition_integral_operators3} yields that there is a 
  constant $\kappa_1$ (independent of $\lambda$) such that
  \begin{equation} \label{estimate1}
    \| T_1(\lambda) \| \leq \kappa_1 \left( m - \frac{\lambda}{c^2} \right) 
    \rightarrow 0, \quad \lambda \nearrow m c^2.
  \end{equation}
  Similarly, as
  $|t_2^\lambda(x)| \leq \kappa_2 \sqrt{(m c)^2 - \lambda^2/c^2}~ |x|^{-1}$ for all $x\in \mathbb{R}^3$ and 
  a constant $\kappa_2$ we obtain from 
  Proposition~\ref{proposition_integral_operators3} a constant $\kappa_3$ (independent of $\lambda$) such that
  \begin{equation} \label{estimate2}
    \| T_2(\lambda) \| \leq \kappa_3 \sqrt{(m c)^2 - \frac{\lambda^2}{c^2}} 
    \rightarrow 0, \quad \lambda \nearrow m c^2.
  \end{equation}  
  Eventually, to get a suitable estimate for $t_3^\lambda$ we note first that 
  \begin{equation*}
    \begin{split}
      \Big|1 - e^{-\sqrt{(m c)^2-\lambda^2/c^2} |x|}\Big| &= 
        \left| \int_{-1}^0 \frac{\mathrm{d}}{\mathrm{d} t} e^{t \sqrt{(m c)^2-\lambda^2/c^2} |x|} \mathrm{d} t \right| \\
      &\leq \int_{-1}^0 \biggl| e^{t \sqrt{(m c)^2-\lambda^2/c^2} |x|} 
          \cdot \sqrt{(m c)^2-\frac{\lambda^2}{c^2}} ~|x| \biggr| \mathrm{d} t \\
      &\leq \sqrt{(m c)^2-\frac{\lambda^2}{c^2}} ~|x|.
    \end{split}
  \end{equation*}
  Thus, there exists a constant $\kappa_4$ such that 
  $|t_3^\lambda(x)| \leq \kappa_4 \sqrt{(m c)^2 - \lambda^2/c^2} ~\big( 1 + |x|^{-1} \big)$ for all $x \in \mathbb{R}^3$. 
  Therefore, we can apply Proposition~\ref{proposition_integral_operators3} and obtain some $\kappa_5$ (independent of $\lambda$) such that
  \begin{equation} \label{estimate3}
    \| T_3(\lambda) \| \leq \kappa_5 \sqrt{(m c)^2 - \frac{\lambda^2}{c^2}} 
    \rightarrow 0, \quad \lambda \nearrow m c^2.
  \end{equation}  
  Combing \eqref{estimate1}--\eqref{estimate3} we conclude 
  \begin{equation*}
    \| C - M(\lambda) \| \leq \| T_1(\lambda) \| + \| T_2(\lambda) \| + \| T_3(\lambda) \| \rightarrow 0,
    \quad \lambda \nearrow m c^2,
  \end{equation*}
  which shows the claim of statement (i).
  
  (ii) In the same way as in \cite[Lemma 3.2]{AMV15} (where the case $c=1$ is treated) 
  one verifies
  \begin{equation*}
    \sup_{\lambda \in (-m c^2, m c^2)} \| M(\lambda) \| < \infty.
  \end{equation*}
  Finally, since $M(m c^2) = \lim_{ \lambda \nearrow  m c^2} M(\lambda)$ and $M(-m c^2) = \lim_{ \lambda \searrow  -m c^2} M(\lambda)$
  by definition it follows that
  \begin{equation*}
    M_0 = \sup_{\lambda \in [-m c^2, m c^2]} \| M(\lambda) \| < \infty.
    \qedhere
  \end{equation*}
\end{proof}

In the following proposition we collect some spectral properties of the Weyl function $M(\cdot)$.
In particular, we give a detailed description of the spectrum of $M(\lambda)$ for $\lambda \in [-m c^2, m c^2]$,
which is needed to prove that the discrete spectrum of the Dirac operator with an electrostatic
$\delta$-shell interaction is finite.
The results are mostly contained in \cite[Lemma 3.2]{AMV16}, but for the convenience of the 
reader we add their proofs here.

\begin{prop} \label{proposition_spectrum_Weyl_function}
Let $\{L^2(\Sigma;\mathbb C^4), \Gamma_0, \Gamma_1 \}$ be the quasi boundary triple in Theorem~\ref{theorem_triple} with corresponding Weyl function $M(\cdot)$.
  Then the following assertions hold. 
  \begin{itemize}
    \item[\rm (i)] For all $\lambda \in \rho(A_0)$ there exists a compact operator $K(\lambda)$ in $L^2(\Sigma; \mathbb{C}^4)$ such that
    \begin{equation*}
      M(\lambda)^2 = \frac{1}{4 c^2} I_4 + K(\lambda).
    \end{equation*}
    \item[\rm (ii)] Let $M_0 := \sup_{\lambda \in [-m c^2, m c^2]} \| M(\lambda) \|$.
    Then, there exists an at most countable family of 
    continuous and non-decreasing functions $\mu_n: [-m c^2, m c^2] \rightarrow \big[ \frac{1}{4 c^2 M_0}, M_0 \big]$ 
    such that 
    \begin{equation*}
      \sigma(M(\lambda)) = \left\{ \pm \frac{1}{2 c} \right\} \cup \{ \mu_n(\lambda): n \in \mathbb{N} \} 
          \cup \left\{ -\frac{1}{4 c^2 \mu_n(\lambda)}: n \in \mathbb{N} \right\}.
    \end{equation*}
    Moreover, for any fixed $\lambda \in [-m c^2, m c^2]$ the number $\frac{1}{2 c}$ is the only 
    possible accumulation point of the sequence $(\mu_n(\lambda))$.
  \end{itemize}
\end{prop}
\begin{proof}
  (i) First, it follows from \cite[equation (22) and Lemma 3.5]{AMV14} that
  \begin{equation*}
    M(0)^2 = \frac{1}{4 c^2} I_4 + K,
  \end{equation*}
  where $K$ is a compact operator in $L^2(\Sigma; \mathbb{C}^4)$ (note that $c M(0) = C_\sigma$ in the 
  notation of \cite[Lemma 3.1 and Lemma 3.3]{AMV14}, where $m$ in \cite[Lemma~3.1]{AMV14} has to be replaced by $m c$).
  For  $\lambda \in \rho(A_0)$ we have 
  \begin{equation*}
    M(\lambda) = M(0) + \lambda \gamma(0)^* \gamma(\lambda)
  \end{equation*}
  by  \eqref{holladi}, and as
  all operators on the right hand side are bounded and everywhere defined we get
  \begin{equation*}
    \begin{split}
      M(\lambda)^2 &= M(0)^2 + \lambda M(0) \gamma(0)^* \gamma(\lambda) 
          + \lambda \gamma(0)^* \gamma(\lambda) M(0) + \big( \lambda \gamma(0)^* \gamma(\lambda) \big)^2 \\
      &= \frac{1}{4 c^2} I_4 + K(\lambda),
    \end{split}
  \end{equation*}
  where
  \begin{equation*}
    K(\lambda) := K + \lambda M(0) \gamma(0)^* \gamma(\lambda) 
      + \lambda \gamma(0)^* \gamma(\lambda) M(0) + \big( \lambda \gamma(0)^* \gamma(\lambda) \big)^2
  \end{equation*}
  is compact, as $\gamma(0)^*$ and $\gamma(\lambda)$ are compact by Proposition \ref{proposition_gamma_Weyl}~(i).
  Hence, assertion~(i) of this proposition is true.

  In order to show (ii) assume first that $\lambda \in (-m c^2, m c^2)$.
  By (i) there exist at most countable sequences of eigenvalues $\mu_n^+(\lambda) \subset [0, \infty)$
  and $\mu_n^-(\lambda) \subset (-\infty, 0)$ such that
  \begin{equation*}
    \sigma(M(\lambda)) \subset \left\{ \pm \frac{1}{2 c} \right\} \cup \big\{ \mu_n^+(\lambda): n \in \mathbb{N} \big\}
                               \cup \big\{ \mu_n^-(\lambda): n \in \mathbb{N} \big\}
  \end{equation*}
  and the only possible accumulation point of $\mu_n^\pm(\lambda)$ is $\pm \frac{1}{2 c}$.
  Since $\lambda\mapsto M(\lambda)$ is analytic and monotonously increasing on the interval $(-m c^2, m c^2)$ according to \eqref{bittesehr} 
  the functions
  $\mu_n^\pm: (-m c^2, m c^2) \rightarrow \mathbb{R}$ can be chosen to be continuous and non-decreasing.
  In the proof of \cite[Theorem 3.3]{AMV15} (observe that the operator $C_\sigma^\lambda$ in \cite[Theorem 3.3]{AMV15} coincides 
  with $c M(\lambda)$, when $m$ in \cite{AMV15} is replaced by $mc $) it is shown that 
  \begin{equation*}
    \mu \in \sigma_{\mathrm{p}}(c M(\lambda)) \Leftrightarrow -\frac{1}{4 \mu} \in \sigma_{\mathrm{p}}(c M(\lambda)),
  \end{equation*}
  and hence
  \begin{equation*}
    \mu \in \sigma_{\mathrm{p}}(M(\lambda)) \Leftrightarrow -\frac{1}{4 c^2 \mu} \in \sigma_{\mathrm{p}}(M(\lambda)).
  \end{equation*}
  Thus, it follows that 
  \begin{equation*}
    \mu_n(\lambda):=\mu_n^+(\lambda) \in \left[\frac{1}{4 c^2 M_0}, M_0 \right]\quad\text{and}\quad \mu_n^-(\lambda) = -\frac{1}{4 c^2 \mu_n(\lambda)}.
  \end{equation*}
  In particular, both points
  $\pm \frac{1}{2 c}$ belong to $\sigma(M(\lambda))$ (they are accumulation points
  of $\mu_n^\pm(\lambda)$ or eigenvalues).
  Finally, since the operators $M(\pm m c^2)$ are the continuous extensions of 
  $M(\lambda)$, $\lambda \in (-m c^2, m c^2)$, in the operator norm (see Proposition~\ref{proposition_Weyl_function_mcsquare}~(i)) it follows that 
  the spectral properties of $M(\lambda)$ 
  extend by continuity to the endpoints $\pm mc^2$; cf. \cite[Satz 9.24]{W00}.
\end{proof}

\section{Dirac operators with $\delta$-shell interactions and their spectra}
\label{section_def_A_eta}

In this section we define Dirac operators with electrostatic $\delta$-shell interactions 
supported on surfaces in $\mathbb R^3$ and study their spectral properties.
The definition of the operator $A_\eta$ for constant interaction strength $\eta \neq \pm 2 c$ 
is via the quasi boundary triple in Theorem~\ref{theorem_triple}.

\begin{definition} \label{definition_delta_op}
Let $T$ be given by \eqref{def_T} and let $\{L^2(\Sigma;\mathbb C^4), \Gamma_0, \Gamma_1\}$ be the quasi boundary triple 
in Theorem~\ref{theorem_triple}.
The Dirac operator $A_\eta$ with an electrostatic $\delta$-shell interaction of strength 
$\eta \in \mathbb{R} \setminus \{ \pm 2 c \}$ supported on $\Sigma$ is defined by
\begin{equation*}
  A_\eta := T\upharpoonright \ker(\Gamma_0 + \eta \Gamma_1),
\end{equation*}
or, equivalently, admits the following more explicit representation:
\begin{equation*}
  A_\eta (f + \gamma \varphi) = A_0 f,
  \quad \dom A_\eta = \big\{ f + \gamma \varphi \in \dom T: \eta (f|_\Sigma + M \varphi) = -\varphi \big\}.
\end{equation*}
\end{definition}

The boundary condition for $f + \gamma \varphi \in \dom A_\eta$
corresponds to a certain jump condition:

\begin{remark}
  \rm Let $\Omega \subset \mathbb{R}^3$ be the bounded $C^\infty$-domain with $\partial \Omega = \Sigma$, 
  denote the outer unit normal vector field of $\Omega$ by $\nu$
  and let $h := f + \gamma \varphi \in \dom A_\eta$. It is known that for $x \in \Sigma$
  the nontangential limits
  \begin{equation*}
    h_+(x) := \lim_{\Omega \ni y \rightarrow x} h(y) = f(x) + M \varphi(x) - \frac{i}{2 c} \alpha \cdot \nu ~\varphi(x) 
  \end{equation*}
  and
  \begin{equation*}
    h_-(x) := \lim_{\mathbb{R}^3 \setminus \overline\Omega \ni y \rightarrow x} h(y)
        = f(x) + M \varphi(x) + \frac{i}{2 c} \alpha \cdot \nu ~\varphi(x)
  \end{equation*}
  exist and define functions in $L^2(\Sigma; \mathbb{C}^4)$; cf. \cite[Lemma~3.3]{AMV14}
  (note that $c \gamma = \Phi (\cdot)$ and $c M = C_\sigma$ with $\Phi(\cdot)$ 
  and $C_\sigma$ in the notation of \cite[Lemma~3.3]{AMV14}).
  Making use of $(\alpha \cdot \nu)^2 = I_4$ (this follows from~\eqref{eq:commutation})
  one verifies that the boundary condition $\eta (f|_\Sigma + M \varphi) = -\varphi$ is equivalent to
  the jump condition
  \begin{equation*}
    \frac{\eta}{2} \left( h_+ + h_- \right) = -i c \alpha \cdot \nu \left( h_+ - h_- \right).
  \end{equation*}
\end{remark}

Note that Green's identity for the quasi boundary triple  $\{L^2(\Sigma;\mathbb C^4), \Gamma_0, \Gamma_1\}$ shows
that $A_\eta$ is symmetric; cf. \eqref{abab}.
In the following we shall employ some abstract results on quasi boundary triples 
and their Weyl functions from Section~\ref{section_quasi_boundary_triples},
which together with the properties of the $\gamma$-field and Weyl function 
$M(\cdot)$ in Propositions~\ref{proposition_gamma_Weyl}--\ref{proposition_spectrum_Weyl_function} 
are the main ingredients in the proofs of Theorem~\ref{theorem_Krein} and
Theorem~\ref{theorem_resolvent_power_difference} below. We first verify that $I_4 + \eta M(\lambda)$ is boundedly invertible.

\begin{lem} \label{lemma_inv2}
  Let $\eta \in \mathbb{R} \setminus \{ \pm 2 c \}$ and let $\lambda \in \mathbb{C} \setminus \mathbb{R}$.
  Then, $I_4 + \eta M(\lambda)$ has a bounded and everywhere defined inverse.
\end{lem}
\begin{proof}
  First, we note that $I_4 + \eta M(\lambda)$ is injective for $\lambda \in \mathbb{C} \setminus \mathbb{R}$,
  as otherwise $\lambda$ would be a non-real eigenvalue of the symmetric operator $A_\eta$; 
  cf. Theorem~\ref{theorem_Krein_abstract}.
  It remains to prove that $I_4 + \eta M(\lambda)$ is surjective. Observe that by  
  Proposition~\ref{proposition_spectrum_Weyl_function}~(i)
  \begin{equation*}
    (I_4 + \eta M(\lambda)) (I_4 - \eta M(\lambda)) = I_4 - \eta^2 M(\lambda)^2
    = \left(1 - \frac{\eta^2}{4 c^2}\right) I_4 - \eta^2 K(\lambda),
  \end{equation*}
  where $K(\lambda)$ is a compact operator. Hence,
  \begin{equation*}
      (I_4 + \eta M(\lambda)) (I_4 - \eta M(\lambda)) = \left(1 - \frac{\eta^2}{4 c^2}\right) ( I_4 + d K(\lambda)),
      \quad d=-\frac{4 c^2 \eta^2}{4 c^2 - \eta^2},
  \end{equation*}
  and therefore $\ran(I_4 + d K(\lambda)) \subset \ran (I_4 + \eta M(\lambda))$.
  Since the left hand side in the above equation is injective
  (otherwise $\lambda$ would be a non-real eigenvalue of one the symmetric operators $A_{\pm \eta}$
  by Theorem~\ref{theorem_Krein_abstract})
  the same is true for the right hand side. Thus, the Fredholm
  alternative implies that $\ran(I_4 + d K(\lambda)) = L^2(\Sigma; \mathbb{C}^4)$.
  Hence, $I_4 + \eta M(\lambda)$ is also surjective, which yields the assertion.
\end{proof}

In the next theorem we verify the self-adjointness of $A_\eta$, provide a Krein type resolvent formula, and 
we investigate the discrete spectrum of $A_\eta$ in the gap $(-mc^2,mc^2)$
of the essential spectrum. It turns out in (iii) that the discrete spectrum in $(-mc^2,mc^2)$ is finite 
(and non-trivial for certain $\eta$ by Corollary~\ref{cor_large_eta}).
Moreover, for sufficiently small and sufficiently large $|\eta|$ the discrete spectrum of $A_\eta$ is empty by assertion~(iv).
While this behavior for small interaction strengths is similar as for Schr\"odinger operators with
$\delta$-interactions, such an effect does not occur for large $\eta$. This result and also 
assertion (ii) are known from \cite{AMV15}; here they follow immediately from Theorem~\ref{theorem_Krein_abstract} 
and Proposition~\ref{proposition_spectrum_Weyl_function}.

\begin{thm} \label{theorem_Krein}
Let $\{L^2(\Sigma;\mathbb C^4), \Gamma_0, \Gamma_1 \}$ be the quasi boundary triple in Theorem~\ref{theorem_triple} 
with corresponding $\gamma$-field $\gamma(\cdot)$ and 
Weyl function $M(\cdot)$. As in Proposition~\ref{proposition_Weyl_function_mcsquare}~{\rm (ii)} set 
$$M_0 := \sup_{\lambda \in [-m c^2, m c^2]} \| M(\lambda) \|.$$
Then the Dirac operator
  $A_\eta$ in Definition~\ref{definition_delta_op} is self-adjoint in $L^2(\mathbb R^3;\mathbb C^4)$ 
  for all $\eta \in \mathbb{R} \setminus \{ \pm 2 c \}$ and
    \begin{equation}\label{resform}
      (A_\eta - \lambda)^{-1} = (A_0 - \lambda)^{-1}
          - \gamma(\lambda) \big( I_4 + \eta M(\lambda) \big)^{-1} \eta \gamma(\overline{\lambda})^*
    \end{equation}
    for all $\lambda \in \rho(A_0)\cap\rho(A_\eta)$. Furthermore, the following assertions are true.
  \begin{itemize}
    \item[\rm (i)] $\sigma_{\mathrm{ess}}(A_\eta) = (-\infty, -m c^2] \cup [m c^2, \infty)$.
    \item[\rm (ii)] $\dim\ker(A_\eta-\lambda)=\dim\ker(I_4+\eta M (\lambda))$ for all $\lambda \in (-m c^2, m c^2)$.
    \item[\rm (iii)]
  $\sigma(A_\eta)\cap(-mc^2,mc^2)$ is finite for all $\eta \in \mathbb{R} \setminus \{ \pm 2 c \}$.
    \item[\rm (iv)] $\sigma(A_\eta)\cap(-mc^2,mc^2)=\emptyset$ either for $\vert\eta\vert < \frac{1}{M_0}$ or for 
    $\vert\eta\vert> 4 c^2 M_0$.
\end{itemize}
  \end{thm}
\begin{proof}
The fact that $A_\eta$ is self-adjoint in $L^2(\mathbb R^3;\mathbb C^4)$ and that the resolvent of $A_\eta$ is given by \eqref{resform} are immediate consequences
of Theorem~\ref{theorem_Krein_abstract}
and Lemma~\ref{lemma_inv2}.

 (i) The resolvent formula \eqref{resform} implies that
  $(A_\eta - \lambda)^{-1} - (A_0 - \lambda)^{-1}$ is compact for all $\lambda\in\rho(A_0)\cap\rho(A_\eta)$
  since $\gamma(\lambda)$ and $\gamma\big(\overline{\lambda}\big)^*$ are compact
  by Proposition~\ref{proposition_gamma_Weyl} and $(I_4 + \eta M(\lambda) )^{-1}\eta$ is bounded by Lemma~\ref{lemma_inv2}.
  This yields
  \begin{equation*}
    \sigma_{\mathrm{ess}}(A_\eta) = \sigma_{\mathrm{ess}}(A_0) =\sigma(A_0)= (-\infty, -m c^2] \cup [m c^2, \infty).
  \end{equation*}

  (ii) This claim follows from Theorem~\ref{theorem_Krein_abstract}.

  Assertion (iii) will be shown by an indirect proof. Assume that for some 
  interaction strength $\eta \in \mathbb{R} \setminus \{ \pm 2 c \}$ 
  there are infinitely many discrete eigenvalues of $A_\eta$ in the gap
  $(-mc^2,mc^2)$ of the essential spectrum. Then $mc^2$ or $-mc^2$ is an accumulation point 
  and in the following we discuss the case $\eta<0$ and that there is a sequence
  $(\lambda_n)\subset\sigma(A_\eta)\cap(-mc^2,mc^2)$ which tends to $mc^2$;
  the cases with $\eta>0$ or eigenvalues accumulating to $-mc^2$ can be treated analogously.
  Recall from Proposition~\ref{proposition_spectrum_Weyl_function}~(ii) that
  \begin{equation*}
    \sigma(M(\lambda)) = \left\{ \pm \frac{1}{2 c} \right\} \cup \{ \mu_n(\lambda): n \in \mathbb{N} \} 
        \cup \left\{ -\frac{1}{4 c^2 \mu_n(\lambda)}: n \in \mathbb{N} \right\},
  \end{equation*}
  where $\mu_n: [-m c^2, m c^2] \rightarrow \big[ \frac{1}{4 c^2 M_0}, M_0 \big]$
  are continuous and non-decreasing functions.
  Since $0<-\frac{1}{\eta}\in\sigma_{\mathrm{p}}(M(\lambda_n))$ by (ii) and
  $-\frac{1}{\eta}\not= \frac{1}{2c}$ for each $n\in\mathbb N$ there exists $k(n)$ such that
  $\mu_{k(n)}(\lambda_n)=-\frac{1}{\eta}$. By monotonicity we have 
  \begin{equation*}
    \frac{1}{4 c^2 M_0} \leq \mu_{k(n)}(-mc^2) \leq -\frac{1}{\eta}\quad\text{and}\quad 
    -\frac{1}{\eta} \leq \mu_{k(n)}(mc^2)\leq M_0
  \end{equation*}
  for all $n\in\mathbb N$
  and hence the infinite sequences $(\mu_{k(n)}(-mc^2))\subset\sigma(M(-mc^2))$ and 
  $(\mu_{k(n)}(mc^2))\subset\sigma(M(mc^2))$
  both have an accumulation point in $\big[\frac{1}{4 c^2 M_0},-\frac{1}{\eta}\big]$ 
  and $\big[-\frac{1}{\eta},M_0\big]$, respectively.
  Since $\frac{1}{2c}$ is the only possible accumulation
  point of $\sigma(M(-mc^2))$ and $\sigma(M(mc^2))$ in $\big[\frac{1}{4 c^2 M_0}, M_0\big]$ this is a contradiction
  to $\eta \neq -2 c$. 
  It follows that $\sigma(A_\eta)\cap(-mc^2,mc^2)$ is finite.

  (iv) For $\eta\not\in\{0,\pm 2c\}$ it follows from (ii) that $\lambda\in (-mc^2,mc^2)$ is an eigenvalue of $A_\eta$
  if and only if  $-\frac{1}{\eta}$ is an eigenvalue of $M(\lambda)$. Hence the assertion follows
  from the fact that 
  $\sigma(M(\lambda)) \subset [ -M_0, -\frac{1}{4 c^2 M_0} ] \cup [ \frac{1}{4 c^2 M_0}, M_0 ]$,
  see Proposition~\ref{proposition_spectrum_Weyl_function}~(ii).
\end{proof}

Besides the qualitative properties of the spectrum of $A_\eta$ in Theorem~\ref{theorem_Krein} we establish a trace class result
important for mathematical scattering theory in Theorem~\ref{theorem_resolvent_power_difference} below.
We keep the notations simple and skip the respective spaces in the symbols of (weak)
Schatten-von Neumann ideals $\mathfrak S_{p,\infty}$. We also note the useful property
\begin{equation}\label{hurra}
\mathfrak S_{1/x,\infty}\cdot\mathfrak S_{1/y,\infty}=\mathfrak S_{1/(x+y),\infty},\qquad x,y>0,
\end{equation}
see, e.g. \cite[Lemma 2.3~(iii)]{BLL13_1}.
In the next preparatory lemma we first provide some useful Schatten-von Neumann estimates for the derivatives 
of the $\gamma$-field and Weyl function
in Proposition~\ref{proposition_gamma_Weyl}.

\begin{lem} \label{lemma_compact_gamm_M}
  Let $\lambda \in \mathbb{C} \setminus \mathbb{R}$
  and let the operators $\gamma(\lambda)$ and $M(\lambda)$
  be given as in Proposition \ref{proposition_gamma_Weyl}. Then for all $k\in\mathbb{N}_0$ one has
  \begin{equation*}
   \frac{\mathrm{d}^k}{\mathrm{d} \lambda^k} \gamma(\lambda)  \in \mathfrak{S}_{4/(2k+1), \infty},
   \quad\text{and}\quad
   \frac{\mathrm{d}^k}{\mathrm{d} \lambda^k} \gamma\big(\overline\lambda\big)^* \in \mathfrak{S}_{4/(2k+1), \infty}.
  \end{equation*}
  Moreover, it holds for all $k \in \mathbb{N}$ 
  \begin{equation*}
       \frac{\mathrm{d}^k}{\mathrm{d} \lambda^k} M(\lambda) \in \mathfrak{S}_{2/k, \infty}.
  \end{equation*}
\end{lem}
\begin{proof}
We shall use that for all $\lambda\in\rho(A_0)$ the relations
\begin{equation}\label{klar}
 \frac{d^k}{d\lambda^k}\gamma\big(\overline\lambda\big)^*=k!\Gamma_1 (A_0-\lambda)^{-k-1},\qquad k=0,1,\dots,
\end{equation}
 and
 \begin{equation}\label{klar2}
 \frac{d^k}{d\lambda^k} M(\lambda)=k!\Gamma_1(A_0-\lambda)^{-k}\gamma(\lambda),\qquad k=1,2,\dots,
\end{equation}
hold; see \eqref{jaha} and \eqref{spitzenklasse}. 
It follows from \eqref{A_0_square} and $\dom\Delta^l=H^{2l}(\mathbb R^3; \mathbb{C})$
that $\dom A_0^{k+1}=H^{k+1}(\mathbb R^3;\mathbb{C}^4)$ and hence 
$\ran(A_0-\lambda)^{-k-1}=H^{k+1}(\mathbb R^3; \mathbb C^4)$. 
Therefore, $\ran(\Gamma_1(A_0-\lambda)^{-k-1}) =  H^{k+1/2}(\Sigma; \mathbb{C}^4)$
  and \cite[Lemma 4.7]{BLL13_1} yields
  \begin{equation}\label{hoho}
    \Gamma_1(A_0-\lambda)^{-k-1} \in \mathfrak S_{4/(2k+1),\infty},\qquad k=0,1,\dots.
  \end{equation}
  Now the second assertion of the lemma follows from \eqref{klar} and
  taking adjoint shows the first statement.
  The assertion on $\frac{\mathrm{d}^k}{\mathrm{d} \lambda^k}M(\lambda)$
  follows from \eqref{klar2}, \eqref{hoho}, $\gamma(\lambda)\in\mathfrak S_{4,\infty}$ and \eqref{hurra}.
\end{proof}

In the next theorem we prove that the difference of the
third powers of the resolvents of $A_\eta$ and $A_0$ is a trace class operator, and 
we provide a formula for the trace in terms of the Weyl function $M(\cdot)$. Note that the trace on the left hand side in \eqref{tftf} 
is taken in the space $L^2(\mathbb R^3;\mathbb C^4)$, whereas the trace on the right hand side is in the boundary space $L^2(\Sigma;\mathbb C^4)$.
We refer the reader to \cite{BLL13_2, GHSZ96, GMZ07} and the references therein for related results on elliptic differential operators, 
Fredholm perturbation determinants and other types of trace formulae for Schr\"{o}dinger operators.

\begin{thm} \label{theorem_resolvent_power_difference}
  Let $\eta \in \mathbb{R} \setminus \{ \pm 2 c \}$ and let $M(\cdot)$ be as in Proposition \ref{proposition_gamma_Weyl}. Then for all $\lambda \in \rho(A_0)\cap\rho(A_\eta)$ the operator
  \begin{equation*}
    (A_\eta - \lambda)^{-3} - (A_0 - \lambda)^{-3}
  \end{equation*}
  belongs to the trace class ideal and
  \begin{equation}\label{tftf}
    \mathrm{tr}\left[ (A_\eta - \lambda)^{-3} - (A_0 - \lambda)^{-3} \right]=
        -\frac{1}{2} \mathrm{tr}\left[ \frac{\mathrm{d}^{2}}{\mathrm{d} \lambda^{2} }
        \left( (I_4 + \eta M(\lambda))^{-1} \eta \frac{\mathrm{d}}{\mathrm{d}\lambda}M(\lambda) \right) \right]
  \end{equation}
  holds.
  In particular,  the wave operators for the pair $\{ A_\eta, A_0 \}$ exist and are complete, and
  the absolutely continuous parts of $A_\eta$ and $A_0$ are unitarily equivalent.
\end{thm}

\begin{proof}
For
$\eta \in \mathbb{R} \setminus \{ \pm 2 c \}$ and $\lambda \in \rho(A_0)\cap\rho(A_\eta)$ 
it follows from Lemma~\ref{lemma_inv2} and Theorem~\ref{theorem_Krein} that
$(I_4 + \eta M(\lambda))^{-1}\eta$ is a bounded and everywhere defined operator. We shall use the symbol $\mathfrak B$ for the class of bounded and every defined operators
in the following.
  The resolvent formula from Theorem \ref{theorem_Krein} and
  \cite[equation~(2.7)]{BLL13_2} yield
  \begin{equation} \label{second_derivative}
    \begin{split}
      &(A_\eta - \lambda)^{-3} - (A_0 - \lambda)^{-3} \\
      &\quad = \frac{1}{2} \frac{\mathrm{d}^2 }{\mathrm{d} \lambda^2}
          \left[ (A_\eta - \lambda)^{-1} - (A_0 - \lambda)^{-1} \right] \\
      &\quad= -\frac{1}{2} \frac{\mathrm{d}^2 }{\mathrm{d} \lambda^2}
          \left[ \gamma(\lambda) (I_4 + \eta M(\lambda))^{-1} \eta \gamma\big(\overline{\lambda}\big)^* \right] \\
      &\quad= -\sum_{p + q + r = 2} \frac{1}{p! q! r!} \left(\frac{\mathrm{d}^p }{\mathrm{d} \lambda^p} \gamma(\lambda)\right)
          \left(\frac{\mathrm{d}^q}{\mathrm{d} \lambda^q} (I_4 + \eta M(\lambda))^{-1} \eta\right)
          \left(\frac{\mathrm{d}^r}{\mathrm{d} \lambda^r} \gamma\big(\overline{\lambda}\big)^*\right).
    \end{split}
  \end{equation}
Before we verify that each summand in the right-hand side in
\eqref{second_derivative} is a trace class operator we first mention that 
  \begin{equation*}
      \frac{\mathrm{d}}{\mathrm{d} \lambda} (I_4 + \eta M(\lambda))^{-1}\eta
          = - (I_4 + \eta M(\lambda))^{-1}\eta \left(\frac{\mathrm{d}}{\mathrm{d}\lambda}M(\lambda)\right) 
          (I_4 + \eta M(\lambda))^{-1}\eta \in \mathfrak{S}_{2, \infty}
  \end{equation*}
  and
  \begin{equation*}
   \begin{split}
    \frac{\mathrm{d}^2}{\mathrm{d} \lambda^2} (I_4 + \eta M(\lambda))^{-1}\eta
          &= 2(I_4 + \eta M(\lambda))^{-1}\eta \left(\left(\frac{\mathrm{d}}{\mathrm{d}\lambda} M(\lambda)\right) 
          (I_4 + \eta M(\lambda))^{-1}\eta\right)^2\\
          &~-(I_4 + \eta M(\lambda))^{-1}\eta \left(\frac{\mathrm{d}^2}{\mathrm{d}\lambda^2} M(\lambda)\right) 
          (I_4 + \eta M(\lambda))^{-1}\eta \in \mathfrak{S}_{1, \infty}
   \end{split}
  \end{equation*}
hold by Lemma~\ref{lemma_compact_gamm_M} and \eqref{hurra}.
It then follows from Lemma~\ref{lemma_compact_gamm_M} that
\begin{equation*}
 \begin{split}
  \left(\frac{\mathrm{d}^2 }{\mathrm{d} \lambda^2} \gamma(\lambda)\right)
      (I_4 + \eta M(\lambda))^{-1} \eta
      \gamma\big(\overline{\lambda}\big)^* &\in \mathfrak{S}_{4/5, \infty}\cdot\mathfrak B \cdot\mathfrak S_{4,\infty},\\
  \left(\frac{\mathrm{d} }{\mathrm{d} \lambda} \gamma(\lambda)\right)
      \left(\frac{\mathrm{d}}{\mathrm{d}\lambda }(I_4 + \eta M(\lambda))^{-1} \eta\right)
      \gamma\big(\overline{\lambda}\big)^* &\in \mathfrak{S}_{4/3, \infty}\cdot\mathfrak S_{2,\infty}\cdot\mathfrak S_{4,\infty},\\
  \left(\frac{\mathrm{d} }{\mathrm{d} \lambda} \gamma(\lambda)\right)
      (I_4 + \eta M(\lambda))^{-1} \eta
  \left(\frac{\mathrm{d}}{\mathrm{d}\lambda } \gamma\big(\overline{\lambda}\big)^*\right) &\in \mathfrak{S}_{4/3, \infty}\cdot\mathfrak B\cdot\mathfrak S_{4/3,\infty},\\
   \gamma(\lambda)
   \left(\frac{\mathrm{d} }{\mathrm{d} \lambda}   (I_4 + \eta M(\lambda))^{-1} \eta\right)
    \left(\frac{\mathrm{d}}{\mathrm{d}\lambda } \gamma\big(\overline{\lambda}\big)^*\right) &\in \mathfrak{S}_{4, \infty}\cdot\mathfrak S_{2,\infty}\cdot\mathfrak S_{4/3,\infty},\\
  \gamma(\lambda)
      (I_4 + \eta M(\lambda))^{-1} \eta \left(\frac{\mathrm{d}^2 }{\mathrm{d} \lambda^2}
      \gamma\big(\overline{\lambda}\big)^*\right) &\in \mathfrak{S}_{4, \infty}\cdot\mathfrak B\cdot\mathfrak S_{4/5,\infty},     \\
   \gamma(\lambda)
      \left(\frac{\mathrm{d}^2}{\mathrm{d}\lambda^2 }(I_4 + \eta M(\lambda))^{-1} \eta\right)
      \gamma\big(\overline{\lambda}\big)^* & \in \mathfrak{S}_{4, \infty}\cdot\mathfrak S_{1,\infty}\cdot\mathfrak S_{4,\infty},\\
 \end{split}
\end{equation*}
and using \eqref{hurra} we observe that each term is in the weak Schatten--von Neumann ideal $\mathfrak S_{2/3,\infty}$.
Since $\mathfrak S_{2/3,\infty}$ is contained in the trace class ideal 
we then conclude from \eqref{second_derivative} the first claim of this theorem.
Moreover, using the cyclicity of the trace 
it follows in the same way as in the proof of \cite[Theorem~3.7~(ii)]{BLL13_2} from \eqref{second_derivative} that
\begin{equation*}
 \begin{split}
  &\mathrm{tr}\bigl((A_\eta - \lambda)^{-3} - (A_0 - \lambda)^{-3}\bigr) \\
  &\qquad= -\sum_{p + q + r = 2} \frac{1}{p! q! r!} \mathrm{tr}\left[\left(\frac{\mathrm{d}^p }{\mathrm{d} \lambda^p} 
          \gamma(\lambda)\right)
          \left(\frac{\mathrm{d}^q}{\mathrm{d} \lambda^q} (I_4 + \eta M(\lambda))^{-1} \eta\right)
          \left(\frac{\mathrm{d}^r}{\mathrm{d} \lambda^r} \gamma\big(\overline{\lambda}\big)^*\right)\right] \\
      &\qquad = -\sum_{p + q + r = 2} \frac{1}{p! q! r!} \mathrm{tr}\left[
          \left(\frac{\mathrm{d}^q}{\mathrm{d} \lambda^q} (I_4 + \eta M(\lambda))^{-1} \eta\right)
          \left(\frac{\mathrm{d}^r}{\mathrm{d} \lambda^r} \gamma\big(\overline{\lambda}\big)^*\right)
          \left(\frac{\mathrm{d}^p }{\mathrm{d} \lambda^p}\gamma(\lambda)\right)\right]\\
      &\qquad= -\frac{1}{2} \mathrm{tr}\left[\frac{\mathrm{d}^2 }{\mathrm{d} \lambda^2}
          \left( (I_4 + \eta M(\lambda))^{-1} \eta \gamma\big(\overline{\lambda}\big)^*\gamma(\lambda)\right) \right] \\
          &\qquad= -\frac{1}{2} \mathrm{tr}\left[\frac{\mathrm{d}^2 }{\mathrm{d} \lambda^2}
          \left( (I_4 + \eta M(\lambda))^{-1} \eta \frac{\mathrm{d}}{\mathrm{d}\lambda} M(\lambda) \right)\right].
 \end{split}
\end{equation*}
This shows the trace formula in Theorem~\ref{theorem_resolvent_power_difference}. 
The assertion on the wave operators and the absolutely continuous spectrum
are well-known consequences of the trace class property, see, e.g~\cite[Chapter~0, Theorem~8.2]{Y10}, \cite[Problem~25]{RS79}, 
and the standard definition of existence and completeness of wave operators.
\end{proof}

\section{The nonrelativistic limit} \label{section_nonrelativistic_limit}

In this section we show that the Dirac operator $A_\eta$ with an electrostatic 
$\delta$-shell interaction of strength
$\eta \in \mathbb{R}$ converges in the nonrelativistic limit, i.e. when the energy of the rest mass $m c^2$ 
is subtracted from the total energy
and the speed of light $c$ tends to $\infty$,
to a Schr\"odinger operator with an electric $\delta$-potential of strength~$\eta$. 
This shows that $A_\eta$ is indeed the relativistic counterpart of the Schr\"odinger operator with a $\delta$-interaction.
Because of the convergence in the nonrelativistic limit one can also deduce spectral properties of $A_\eta$
for large $c$ from those of the Schr\"odinger operator with a $\delta$-interaction. 
As an illustration we show in Corollary~\ref{cor_large_eta} that for sufficiently large $-\eta>0$ 
the number of eigenvalues of $A_\eta$ in the gap $(- m c^2, m c^2)$ 
of $\sigma_{\text{ess}}(A_\eta)$ becomes large.

First we recall the definition of the Schr\"odinger operator with a $\delta$-potential supported on $\Sigma$ 
of strength $\eta\in\mathbb R$ and some of its properties. For this consider the sesquilinear form
\begin{equation} \label{def_delta_form}
  \mathfrak{b}_{\eta}[f, g] := \frac{1}{2 m} \big( \nabla f, \nabla g \big)_{L^2(\mathbb{R}^3; \mathbb{C}^3)} 
        + \eta (f|_\Sigma, g|_\Sigma)_{L^2(\Sigma; \mathbb{C})}, ~
        \dom \mathfrak{b}_\eta = H^1(\mathbb{R}^3; \mathbb{C}),
\end{equation}
which is symmetric, bounded from below and closed, see \cite[Section~4]{BEKS94} or \cite{BLL13}.
The corresponding self-adjoint operator $-\Delta_\eta$ is the Schr\"odinger operator with a $\delta$-potential supported on 
$\Sigma$ of strength $\eta$.
In what follows, we want to state a suitable resolvent formula for $-\Delta_\eta$. 
For this purpose, we introduce for $\lambda \in \mathbb{C} \setminus \mathbb{R}$ the function 
\begin{equation} \label{def_K_lambda}
  K_\lambda(x) := 2 m \frac{e^{i\sqrt{2 m \lambda} |x|}}{4 \pi |x|}, \quad x \in \mathbb{R}^3 \setminus \{ 0 \}.
\end{equation}
%Recall that $\sqrt{2 m \lambda}$ is understood for $\lambda \in \mathbb{C} \setminus \mathbb{R}$ 
%in such a way that $\mathrm{Im} \sqrt{2 m \lambda} > 0$.
Then, $K_\lambda$ is the integral kernel of the resolvent of $-\frac{1}{2 m} \Delta$, i.e. 
\begin{equation} \label{resolvent_Laplace}
  \left(-\frac{1}{2 m} \Delta - \lambda\right)^{-1} f(x) = \int_{\mathbb{R}^3} K_\lambda(x - y) f(y) \mathrm{d} y, \quad
      x \in \mathbb{R}^3, ~f \in L^2(\mathbb{R}^3; \mathbb{C}).
\end{equation}
Furthermore, we define the operators 
$\widetilde{\gamma}(\lambda): L^2(\Sigma; \mathbb{C}) \rightarrow L^2(\mathbb{R}^3; \mathbb{C})$,
\begin{equation} \label{def_gamma_Schroedinger}
  \widetilde{\gamma}(\lambda) \varphi(x) := \int_\Sigma K_\lambda(x - y) \varphi(y) \mathrm{d} \sigma(y),
  \quad x \in \mathbb{R}^3, ~\varphi \in L^2(\Sigma; \mathbb{C}),
\end{equation}
and $\widetilde{M}(\lambda): L^2(\Sigma; \mathbb{C}) \rightarrow L^2(\Sigma; \mathbb{C})$,
\begin{equation} \label{def_M_Schroedinger}
  \widetilde{M}(\lambda) \varphi(x) := \int_\Sigma K_\lambda(x - y) \varphi(y) \mathrm{d} \sigma(y),
  \quad x \in \Sigma, ~\varphi \in L^2(\Sigma; \mathbb{C}).
\end{equation}
According to \cite[Proposition 3.2 and Remark 3.3]{BLL13} the operators $\widetilde{\gamma}(\lambda)$
and $\widetilde{M}(\lambda)$ are bounded and everywhere defined. It is not difficult to check that the 
adjoint of $\widetilde{\gamma}(\lambda)$ is given by 
$\widetilde{\gamma}(\lambda)^*: L^2(\mathbb{R}^3; \mathbb{C}) \rightarrow L^2(\Sigma; \mathbb{C})$,
\begin{equation} \label{def_gamma_star_Schroedinger}
  \widetilde{\gamma}(\lambda)^* f(x) := \int_{\mathbb{R}^3} K_{\overline{\lambda}}(x - y) f(y) \mathrm{d} y,
  \quad x \in \Sigma, ~f \in L^2(\mathbb{R}^3; \mathbb{C}).
\end{equation}
With these notations we recall a resolvent formula for $-\Delta_\eta$;
cf. \cite[Theorem~3.5]{BLL13} or \cite[Lemma~2.3]{BEKS94}.

\begin{thm} \label{theorem_Schroedinger}
  Let $\eta \in \mathbb{R}$
  and let $\lambda \in \mathbb{C} \setminus \mathbb{R}$.
  Then the operator $I + \eta \widetilde{M}(\lambda)$ has a bounded and everywhere defined inverse
  and 
  \begin{equation*}
    (-\Delta_\eta - \lambda)^{-1} = \left( -\frac{1}{2 m} \Delta - \lambda \right)^{-1}
        - \widetilde{\gamma}(\lambda) \big( I + \eta \widetilde{M}(\lambda) \big)^{-1} 
        \eta \widetilde{\gamma}\big(\overline{\lambda}\big)^*.
  \end{equation*}
\end{thm}

It will be shown that the resolvents of the Dirac operators $A_\eta$ 
with $\eta\in\mathbb R$ fixed converge in the nonrelativistic limit to the resolvent of the Schr\"odinger
operator with a $\delta$-potential times a projection to the upper components of the Dirac wave function, i.e. that
for any $\lambda \in \mathbb{C} \setminus \mathbb{R}$
\begin{equation*}
  \lim_{c \rightarrow \infty} \big( A_\eta - (\lambda + m c^2) \big)^{-1}
      = \big( -\Delta_{\eta} - \lambda \big)^{-1}  P_+,
\end{equation*}
where
\begin{equation*}
  P_+ := \begin{pmatrix} I_2 & 0 \\ 0 & 0 \end{pmatrix}.
\end{equation*}
Note that the Dirac operator $A_\eta$ is self-adjoint for all sufficiently large $c$ by Theorem~\ref{theorem_Krein}.
The resolvent formula in Theorem~\ref{theorem_Krein} indicates that it is sufficient to compute the limits of the operators 
$( A_0 - (\lambda + m c^2))^{-1}$, $\gamma(\lambda + m c^2), M(\lambda + m c^2)$ and
$\gamma\big( \overline{\lambda} + m c^2 \big)^*$.
This is done next in a preparatory proposition. The nonrelativistic limit of the free Dirac operator in \eqref{conv_res_free} is known from 
\cite[Theorem~6.1]{T92}.

\begin{prop} \label{proposition_limit_gamma_M}
  Let $\lambda \in \mathbb{C} \setminus \mathbb{R}$ and let $\gamma(\lambda + m c^2), M(\lambda + m c^2)$ and 
  $\gamma(\overline{\lambda} + m c^2)^*$
  be as in Proposition \ref{proposition_gamma_Weyl}. Moreover, let $\widetilde{\gamma}(\lambda), \widetilde{M}(\lambda)$
  and $\widetilde{\gamma}(\overline{\lambda})^*$ be as in 
  \eqref{def_gamma_Schroedinger}--\eqref{def_gamma_star_Schroedinger}.
  Then there exists a constant $\kappa = \kappa(m, \lambda)$ such that the following statements are true.
  \begin{subequations} \label{nonrelativistic_convergence}
    \begin{gather}
      \label{conv_res_free}
      \left\|\big( A_0 - (\lambda + m c^2) \big)^{-1} -
      \left( -\frac{1}{2 m} \Delta - \lambda \right)^{-1}  P_+ \right\| \leq \frac{\kappa}{c}; \\
      \label{conv_gamma}
      \big\| \gamma(\lambda + m c^2) - \widetilde{\gamma}(\lambda)  P_+ \big\| \leq \frac{\kappa}{c};\\
      \label{conv_gamma_star}
      \big\| \gamma\big(\overline{\lambda} + m c^2\big)^* 
          - \widetilde{\gamma}\big(\overline{\lambda}\big)^*  P_+ \big\| \leq \frac{\kappa}{c}; \\
      \label{conv_M}
      \big\| M(\lambda + m c^2) - \widetilde{M}(\lambda)  P_+ \big\| \leq \frac{\kappa}{c}. 
    \end{gather}
  \end{subequations}
\end{prop}
\begin{proof}
  Since all differences that shall be estimated in the operator norm are integral operators 
  with the integral kernel $G_{\lambda + m c^2} - K_\lambda P_+$ we consider first this function.
%  $G_{\lambda + m c^2}(x) - K_\lambda(x)  P_+$.
  Let $K_\lambda$ be as in \eqref{def_K_lambda} and note that
  \begin{equation*}
%    \begin{split}
      G_{\lambda + m c^2}(x) \!= \left(\frac{\lambda}{c^2} I_4 \!+ 2 m P_+ 
          + \left( 1 - i \sqrt{\frac{\lambda^2}{c^2} + 2 m \lambda} |x| \right) \frac{i(\alpha \cdot x )}{c |x|^2} \right)
     \frac{e^{i\sqrt{\lambda^2/c^2 + 2 m \lambda} |x|}}{4 \pi |x|}.
%    \end{split}
  \end{equation*}
  We use the decomposition
  \begin{equation} \label{T_1_T_2}
    G_{\lambda + m c^2}(x) - K_\lambda(x)  P_+ = t_1(x) + t_2(x),
  \end{equation}
  where the functions $t_1$ and $t_2$ are defined by  
  \begin{equation} \label{def_t_1_t_2}
    \begin{split}
      t_1(x) &= \frac{e^{i\sqrt{\lambda^2/c^2 + 2 m \lambda} |x|}}{4 \pi |x|} \left( \frac{\lambda}{c^2} I_4 + 
          \left( 1 - i \sqrt{\frac{\lambda^2}{c^2} + 2 m \lambda} |x| \right) \frac{i(\alpha\cdot x)}{c |x|^2} \right);\\
      t_2(x) &= \left( e^{i\sqrt{\lambda^2/c^2 + 2 m \lambda} |x|}
          - e^{i\sqrt{2 m \lambda} |x|} \right) \frac{2 m}{4 \pi |x|} P_+. \\
    \end{split}
  \end{equation}
  It is easy to see that there exist positive constants $\kappa_1(m, \lambda)$ and $\kappa_2(m, \lambda)$ independent of $c$ 
  and an $R > 0$ such that 
  \begin{equation} \label{estimate_T_1}
      |t_1(x)| \leq \frac{\kappa_1(m, \lambda)}{c} \begin{cases} |x|^{-2}, &~ |x| < R, \\
                   e^{-\kappa_2(m, \lambda) |x|}, &~|x| \geq R. \end{cases}
  \end{equation}
  In order to estimate $t_2$ note  that
  \begin{equation} \label{estimate_main_theorem} 
    \begin{split}
      \left|e^{i\sqrt{\lambda^2/c^2 + 2 m \lambda} |x|} - e^{i\sqrt{2 m \lambda} |x|} \right| &= 
          \left| \int_0^1 \frac{\mathrm{d}}{\mathrm{d} t} e^{i\sqrt{t \lambda^2/c^2 + 2 m \lambda} |x|} \mathrm{d} t \right| \\
      &\leq  \frac{|x|}{c} \int_0^1 \biggl| e^{i\sqrt{t \lambda^2/c^2 + 2 m \lambda} |x|} 
          \frac{i \lambda^2 }{2 c \sqrt{t \lambda^2/c^2 + 2 m \lambda}} \biggr| \mathrm{d} t.
    \end{split}
  \end{equation}
  Since $\lambda \in \mathbb{C} \setminus \mathbb{R}$ there exist constants 
  $\kappa_3(m, \lambda), \kappa_4(m, \lambda) > 0$ such that for all sufficiently large $c$
  \begin{equation*}
    \left| \frac{i \lambda^2}{2 c \sqrt{t \lambda^2/c^2 + 2 m \lambda}} \right| \leq \kappa_3(m, \lambda)
    \quad \text{and} \quad 
    \mathrm{Re} \left( i\sqrt{t \lambda^2/c^2 + 2 m \lambda} \right) \leq -\kappa_4(m, \lambda)
  \end{equation*}
  hold for all $t \in [0, 1]$. This implies
  \begin{equation} \label{estimate_T_2}
    \begin{split}
      |t_2(x)| &= \bigg|\frac{2 m}{4 \pi |x|} \left( e^{i\sqrt{\lambda^2/c^2 + 2 m \lambda} |x|}
          - e^{i\sqrt{2 m \lambda} |x|} \right) P_+ \bigg| \\
      &\leq \kappa_3(m, \lambda) \frac{2 m}{4 \pi c} e^{-\kappa_4(m, \lambda) |x|}.
    \end{split}
  \end{equation}
  Eventually, because of the estimates \eqref{T_1_T_2}, \eqref{estimate_T_1} and \eqref{estimate_T_2} 
  there exist constants $\kappa_5(m, \lambda), \kappa_6(m, \lambda) > 0$ such that
  \begin{equation} \label{estimate_T_1_T_2}
    \begin{split}
      |G_{\lambda + m c^2}(x) - K_\lambda(x)  P_+| &\leq |t_1(x)| + |t_2(x)| \\
      &\leq \frac{\kappa_5(m, \lambda)}{c} \begin{cases} |x|^{-2}, &~ |x| < R, \\
                   e^{-\kappa_6(m, \lambda) |x|}, &~|x| \geq R. \end{cases} \\
    \end{split}
  \end{equation}
  Now, we are prepared to prove \eqref{conv_res_free}--\eqref{conv_gamma_star}.
  By~\eqref{resolvent_A_0} and
   \eqref{resolvent_Laplace} we have
  \begin{equation*}
    \begin{split}
      \bigg(\big( A_0 - (\lambda + m c^2) \big)^{-1} &- 
          \left( -\frac{1}{2 m} \Delta - \lambda \right)^{-1}  P_+ \bigg) f(x) \\
      &= \int_{\mathbb{R}^3} \big( G_{\lambda + m c^2}(x - y) - K_\lambda(x - y)  P_+ \big) f(y) \mathrm{d} y
    \end{split}
  \end{equation*}
  for $x\in\mathbb R^3$ and $f \in L^2(\mathbb{R}^3; \mathbb{C}^4)$.
  Employing \eqref{estimate_T_1_T_2} and Proposition~\ref{proposition_integral_operators1} we find that
  \begin{equation*}
    \left\| \big( A_0 - (\lambda + m c^2) \big)^{-1} - 
          \left( -\frac{1}{2 m} \Delta - \lambda \right)^{-1}  P_+ \right\| 
          \leq \frac{\kappa_7(m, \lambda)}{c} 
  \end{equation*}
  for some constant $\kappa_7(m, \lambda)$ and hence \eqref{conv_res_free} holds.
  In order to prove \eqref{conv_gamma} recall from Proposition~\ref{proposition_gamma_Weyl}~(i) and \eqref{def_gamma_Schroedinger} that
  \begin{equation*}
    \big(\gamma(\lambda + m c^2) - \widetilde{\gamma}(\lambda)  P_+\big) \varphi(x)
        = \int_\Sigma \big( G_{\lambda + m c^2}(x-y) - K_\lambda(x-y) P_+ \big) \varphi(y) \mathrm{d} \sigma(y)
  \end{equation*}
  for $x \in \mathbb{R}^3$ and $\varphi \in L^2(\Sigma; \mathbb{C}^4)$.
  Here, the asymptotics in
  \eqref{estimate_T_1_T_2} and Proposition~\ref{proposition_integral_operators2} yield
  \begin{equation*}
    \begin{split}
      \big\| \gamma(\lambda + m c^2) 
          &- \widetilde{\gamma}(\lambda)  P_+ \big\| 
      \leq \frac{\kappa_8(m, \lambda, \Sigma)}{c},
    \end{split}
  \end{equation*}
  which is already the claimed estimate.
  Moreover, the relation~\eqref{conv_gamma_star} follows by taking adjoints.
  Finally, we prove $M(\lambda + m c^2) \rightarrow \widetilde{M}(\lambda)  P_+$.
  For that purpose, we use the decomposition
  \begin{equation*}
    \begin{split} 
      \big(M(\lambda + m c^2) &- \widetilde{M}(\lambda)  P_+ \big) \varphi(x) \\
        &= \lim_{\varepsilon \searrow 0} \int_{|x-y|>\varepsilon}
        \big( G_{\lambda + m c^2}(x-y) - K_\lambda(x-y) P_+ \big) \varphi(y) \mathrm{d} \sigma(y) \\
        &= (U_1 + U_2 + U_3 + U_4) \varphi(x),
        \quad x \in \Sigma, ~\varphi \in L^2(\Sigma; \mathbb{C}^4),
    \end{split}
  \end{equation*}
  where for $j\in \{ 1, 2, 3, 4 \}$ the operators
  $U_j: L^2(\Sigma; \mathbb{C}^4) \rightarrow L^2(\Sigma; \mathbb{C}^4)$ are integral operators of the form
  \begin{equation*}
    U_j \varphi(x) := \lim_{\varepsilon \searrow 0} \int_{|x-y|>\varepsilon}
    u_j(x-y)\varphi(y) \mathrm{d} \sigma(y),
    \quad x \in \Sigma,~\varphi \in L^2(\Sigma; \mathbb{C}^4),
  \end{equation*}
  and the functions $u_j$ are given by
  \begin{align*} 
%    \begin{split}
      u_1(x) &:= \frac{e^{i\sqrt{\lambda^2/c^2 + 2 m \lambda} |x|}}{4 \pi |x|} \left( \frac{\lambda}{c^2} I_4 
          + \frac{\alpha \cdot x}{c |x|} \sqrt{\frac{\lambda^2}{c^2} + 2 m \lambda} \right),
      &u_2(x) &:= t_2(x), \\
      u_3(x) &:= \frac{i(\alpha \cdot x)}{4 c \pi |x|^3} \Big(e^{i\sqrt{\lambda^2/c^2 + 2 m \lambda} |x|} - 1\Big) , 
      &u_4(x) &:= \frac{i(\alpha \cdot x)}{4 c \pi |x|^3} ,
%    \end{split}
  \end{align*}
  with $t_2$ as in \eqref{def_t_1_t_2}. Note that $u_1 + u_3 + u_4 = t_1$
  with $t_1$ given by \eqref{def_t_1_t_2}.
  It is easy to see that $|u_1(x)| \leq \frac{\kappa_9(m, \lambda)}{c |x|}$ for some constant $\kappa_{9}(m, \lambda)$ and all $x \in \mathbb{R}^3\setminus \{0\}$, and
  $|u_2(x)| \leq \kappa_3(m, \lambda) \frac{2 m}{4 \pi c}$ for all $x \in \mathbb{R}^3$ by \eqref{estimate_T_2}. Next, we observe that
  \begin{equation*}
    \begin{split}
      \Big|e^{i\sqrt{\lambda^2/c^2 + 2 m \lambda} |x|} - 1\Big| &= 
        \left| \int_0^1 \frac{\mathrm{d}}{\mathrm{d} t} e^{i t\sqrt{\lambda^2/c^2 + 2 m \lambda} |x|} \mathrm{d} t \right| \\
      &\leq \vert x\vert \int_0^1 \bigg| e^{i t \sqrt{\lambda^2/c^2 + 2 m \lambda} |x|} 
          \cdot i\sqrt{\frac{\lambda^2}{c^2} + 2 m \lambda} \bigg| \mathrm{d} t,
    \end{split}
  \end{equation*}
  and hence there exists $\kappa_{10}(m, \lambda)$ such that $|u_3(x)| \leq \frac{\kappa_{10}(m, \lambda)}{c\vert x\vert}$ 
  for all $x \in \mathbb{R}^3\setminus \{0\}$.
  Therefore, we can apply Proposition~\ref{proposition_integral_operators3} and obtain
  \begin{equation*}
    \| U_j \| \leq \frac{\kappa_{11}(m, \lambda)}{c}, \quad j \in \{ 1, 2, 3 \},
  \end{equation*}
  for some constant $\kappa_{11}(m, \lambda)$.
  Eventually, we note that $U_4 = \frac{1}{c} C$, where $C$ is the integral operator with integral kernel
  $c u_4(x-y)=\frac{i(\alpha \cdot (x-y))}{4 \pi |x-y|^3} $; this operator is independent of $c$, everywhere defined and bounded,
  see the proof of \cite[Lemma~3.3]{AMV14}. Therefore, $\| U_4 \| \leq \frac{\kappa_{12}}{c}$.
  This yields finally that
  \begin{equation*}
    \left\| M(\lambda + m c^2) - \widetilde{M}(\lambda)  P_+ \right\| 
        \leq \| U_1 \| + \| U_2 \| + \| U_3 \| + \| U_4 \| \leq \frac{\kappa_{13}(m, \lambda)}{c}
  \end{equation*}
  and completes the proof of \eqref{conv_M}.
\end{proof}

The next theorem is the main result in this section and basically a consequence of the resolvent formulae for $A_\eta$ and 
$-\Delta_\eta$ from Theorem~\ref{theorem_Krein} and Theorem~\ref{theorem_Schroedinger}, respectively, 
and the estimates in  Proposition \ref{proposition_limit_gamma_M}.

\begin{thm} \label{theorem_nonrelativistic_limit}
  Let $\eta \in \mathbb{R}$ 
  and let $A_\eta$ be the Dirac operator with an electrostatic
  $\delta$-shell interaction in Definition~\ref{definition_delta_op}. 
  Furthermore, denote by $-\Delta_\eta$ the Schr\"odinger operator with a $\delta$-interaction of
  strength $\eta$. Then, for any $\lambda \in \mathbb{C} \setminus \mathbb{R}$ 
  there exists a constant $\kappa = \kappa(m, \lambda, \eta)$ such that
  \begin{equation*}
    \left\| \big( A_\eta - (\lambda + m c^2) \big)^{-1}
        - \big( -\Delta_{\eta} - \lambda \big)^{-1} P_+ \right\| \leq \frac{\kappa}{c}.
  \end{equation*}
\end{thm}

\begin{remark}
  For the special case $\eta = 0$ the convergence of the free Dirac operator to the free Laplace operator
  in the nonrelativistic limit is well-known, see e.g. \cite[Theorem~6.1]{T92}, where it is shown that the order of 
  convergence is $\frac{1}{c}$. Hence, the order of convergence in Theorem~\ref{theorem_nonrelativistic_limit}
  is optimal for general interaction strengths $\eta \in \mathbb{R}$.
\end{remark}

\begin{proof}[Proof of Theorem~\ref{theorem_nonrelativistic_limit}]
  First, recall that by Theorem \ref{theorem_Krein} the resolvent of $A_\eta$ is given by
  \begin{equation*}
    \begin{split}
      \big( A_\eta - (\lambda + m c^2) \big)^{-1} &= \big( A_0 - (\lambda + m c^2) \big)^{-1}\\
        &\quad - \gamma(\lambda + m c^2) \big( 1 + \eta M(\lambda + m c^2) \big)^{-1} 
        \eta \big(\gamma\big( \overline{\lambda} +  m c^2 \big) \big)^*.
    \end{split}
  \end{equation*}
  From Proposition \ref{proposition_limit_gamma_M} we know that there exists a constant $\kappa_1 = \kappa_1(m, \lambda)$
  such that
  \begin{gather*}
    \left\|\big( A_0 - (\lambda + m c^2) \big)^{-1} -
      \left( -\frac{1}{2 m} \Delta - \lambda \right)^{-1}  P_+ \right\| \leq \frac{\kappa_1}{c}; \\
    \big\| \gamma(\lambda + m c^2) - \widetilde{\gamma}(\lambda)  P_+ \big\| \leq \frac{\kappa_1}{c};\\
    \big\| \gamma\big(\overline{\lambda} + m c^2\big)^* 
        - \widetilde{\gamma}\big(\overline{\lambda}\big)^*  P_+\big\| \leq \frac{\kappa_1}{c}; \\
    \big\| M(\lambda + m c^2) - \widetilde{M}(\lambda)  P_+\big\| \leq \frac{\kappa_1}{c}. 
  \end{gather*}
  Since the operators $I_4 + \eta M(\lambda + m c^2)$ and $I_4 + \eta \widetilde{M}(\lambda)  P_+$ are boundedly 
  invertible, see Lemma \ref{lemma_inv2} and Proposition \ref{theorem_Schroedinger}, it follows from 
  \cite[Theorem IV 1.16]{K95} that 
  \begin{equation*}
    \left\| \big( I_4 + \eta M(\lambda + m c^2) \big)^{-1}
        - \big( I_4 + \eta \widetilde{M}(\lambda)  P_+ \big)^{-1} \right\| \leq \frac{\kappa_2}{c}
  \end{equation*}
  holds
  for some constant $\kappa_2 = \kappa_2(m, \lambda, \eta)$. Therefore, by using the resolvent formula for $-\Delta_\eta$
  from Proposition \ref{theorem_Schroedinger} we obtain
  \begin{equation*}
    \begin{split}
      \lim_{c \rightarrow \infty} \big( A_\eta &- (\lambda + m c^2) \big)^{-1} = 
          \lim_{c \rightarrow \infty} \Big[\big( A_0 - (\lambda + m c^2) \big)^{-1}\\
      & \qquad \qquad \qquad \qquad - \gamma(\lambda + m c^2) \big( I_4 + \eta M(\lambda + m c^2) \big)^{-1} 
          \eta \gamma\big( \overline{\lambda} + m c^2 \big)^*\Big] \\
      &= \left( -\frac{1}{2 m}\Delta - \lambda\right)^{-1}  P_+
           - \widetilde{\gamma}(\lambda)  P_+ \big( I_4 + \eta \widetilde{M}(\lambda)  P_+ \big)^{-1} \eta
           \widetilde{\gamma}\big( \overline{\lambda} \big)^*  P_+ \\
      &= \big( -\Delta_\eta - \lambda \big)^{-1}  P_+
    \end{split}
  \end{equation*}
 and the order 
 of convergence in the operator norm can be estimated by~$\frac{1}{c}$. 
 This completes the proof of Theorem~\ref{theorem_nonrelativistic_limit}.
\end{proof}

Finally, we show that for large $c$ and $-\eta>0$ sufficiently large the number of eigenvalues of $A_\eta$ 
in the gap $(-m c^2, m c^2)$ of $\sigma_{\text{ess}}(A_\eta)$
is big. 
The proof is based on Theorem~\ref{theorem_nonrelativistic_limit}
and a result from \cite{E03} on the spectrum of $-\Delta_\eta$.
In a similar way, one can derive also other results on the spectrum of $A_\eta$ from the well-known properties of $-\Delta_\eta$.

\begin{cor}\label{cor_large_eta}
  For any fixed $j \in \mathbb{N}$ there exists $\eta<0$ such that $\sharp \sigma_{\textup{d}}(A_\eta) \geq j$ 
  for all sufficiently large $c$.
\end{cor}

\begin{proof}
  Note first that $\sigma_{\text{ess}}(-\Delta_\eta P_+) = \sigma_{\text{ess}}(-\Delta_\eta) \cup \{ 0 \} = [0, \infty)$
  and recall from \cite[Theorem~3.14]{BLL13} that $\sigma_{\text{d}}(-\Delta_\eta P_+) = \sigma_{\text{d}}(-\Delta_\eta)$ is finite.
  For $j \in \mathbb{N}$ fixed \cite[Theorem~2.1]{E03} yields $\sharp \sigma_{\text{d}}(-\Delta_\eta P_+) \geq j$ for some 
  $\eta<0$.
  Next, choose $a < b < 0$ with $\sigma_{\text{d}}(-\Delta_\eta) \subset (a, b)$
  and denote by $E_{-\Delta_\eta P_+}((a, b))$ and $E_{A_\eta - m c^2}((a, b))$
  the spectral projections of $-\Delta_\eta P_+$ and $A_\eta - m c^2$,
  respectively, corresponding to $(a, b)$.
  For $c\rightarrow\infty$ and $\lambda\in\mathbb C\setminus\mathbb R$ Theorem~\ref{theorem_nonrelativistic_limit} yields that the operators $(A_\eta - (\lambda+ m c^2))^{-1}$
  converge 
  to $(-\Delta_\eta - \lambda)^{-1} P_+$. The latter operator is the resolvent of a self-adjoint relation (multivalued operator) and hence
  it follows in the same way as in  	
  \cite[Satz~9.24\,b)]{W00} together with \cite[Satz~2.58\,a)]{W00}
  that for all sufficiently large $c$ 
  the dimensions of the ranges of $E_{-\Delta_\eta P_+}((a, b))$ and $E_{A_\eta - m c^2}((a, b))$ 
  coincide, i.e.
  \begin{equation*}
    \dim \ran E_{A_\eta - m c^2}((a, b)) = \dim \ran E_{-\Delta_\eta P_+}((a, b)) \geq j.
  \end{equation*} 
  Hence, $A_\eta$ has at least $j$ eigenvalues (counted with multiplicities) in the interval 
  $(a + m c^2, b + m c^2) \subset (-m c^2, m c^2)$ for sufficiently large $c$.
\end{proof}

\begin{appendix}
\section{Criteria for the boundedness of integral operators} \label{appa}

In this appendix we discuss the boundedness of integral operators for 
some special
integral kernels. The results are presented such that they can be applied directly
in the main part of the paper.
First we recall the Schur test, which is the 
abstract tool to prove these
results; cf. \cite[Example~III~2.4]{K95} or \cite[Satz~6.9]{W00}
for the case of scalar integral kernels.

\begin{prop} \label{proposition_Schur_test}
   Let $(X, \mu)$ and $(Y, \nu)$ be $\sigma$-finite measure spaces and
   let $t: X \times Y \rightarrow \mathbb{C}^{n \times n}$ be $\mu 
\times \nu$-measurable.
   Assume that there exist measurable functions $t_1, t_2: X \times Y 
\rightarrow [0, \infty)$
   satisfying $|t|^2 \leq t_1 t_2$ almost everywhere and constants 
$\kappa_1, \kappa_2 > 0$ such that
   \begin{equation*}
     \int_{X} t_1(x, y) \mathrm{d} \mu(x) \leq \kappa_1, \quad y \in Y,
     \quad \text{and} \quad
     \int_{Y} t_2(x, y) \mathrm{d} \nu(y) \leq \kappa_2, \quad x \in X.
   \end{equation*}
   Then the operator $T: L^2(Y; \nu; \mathbb{C}^n) \rightarrow L^2(X; 
\mu; \mathbb{C}^n)$,
\begin{equation*}
     T f(x) = \int_{Y} t(x, y) f(y) \mathrm{d} \nu(y), \quad x \in X,~ f 
\in L^2(Y; \nu; \mathbb{C}^n),
   \end{equation*}
   is everywhere defined and bounded with $\| T \|^2 \leq \kappa_1 \kappa_2$.
   In particular, if $(X, \mu) = (Y, \nu)$ and $t_1(x, y) = t_2(y, x)$ 
for all almost $x, y \in X$, then $\| T \| \leq \kappa_1$.
\end{prop}

In the following  the Schur test will be applied
in the cases that $X$ and $Y$
are either $\mathbb{R}^3$ equipped with the Lebesgue measure or $\Sigma$ (the boundary of a $C^\infty$-smooth bounded domain in $\mathbb R^3$) 
equipped with the
associated Hausdorff measure $\sigma$ and where the integral kernels
satisfy
$\mathcal{O}\big( |x - y|^{-s} \big)$ for small $x - y$ and some
suitable $s > 0$.
For that, we need the following integral estimates.

\begin{lem} \label{lemma_integral_estimates}
The following assertions {\rm (i)--(ii)} hold.
    \begin{itemize}
      \item[\rm (i)] Let $\kappa, R > 0$ and $s \in (0, 3)$ and define
      the function
      \begin{equation*}
        \tau(x) := \begin{cases} |x|^{-s},& ~ |x| < R, \\
                                  e^{-\kappa |x|},& ~ |x| \geq R, 
\end{cases}
      \end{equation*}
      for $x \in \mathbb{R}^3\setminus \{0\}$. Then there is a constant 
$K> 0$ such that for all $x \in
\mathbb{R}^3$
      \begin{equation*}
        \int_{\mathbb{R}^3} \tau(x - y) \mathrm{d} y \leq K.
      \end{equation*}
      \item[\rm (ii)] Let $s \in (0, 2)$. Then there is a constant $K$
such that for all $x \in \mathbb{R}^3$
      \begin{equation*}
        \int_\Sigma \big( 1 + |x - y|^{-s} \big) \mathrm{d} \sigma(y) \leq K.
      \end{equation*}
    \end{itemize}
\end{lem}
\begin{proof}
    (i) For $x \in \mathbb{R}^3$ fixed the translation
invariance of the Lebesgue measure shows
    \begin{equation*}
      \begin{split}
        \int_{\mathbb{R}^3} \tau(x - y) \mathrm{d} y
            = \int_{\mathbb{R}^3} \tau(-y) \mathrm{d} y
            =  \int_{B(0, R)} |y|^{-s} \mathrm{d} y
            +\int_{\mathbb{R}^3 \setminus B(0, R)} e^{-\kappa |y|}
\mathrm{d} y,
      \end{split}
    \end{equation*}
   where the integrals on the right hand side are independent of $x$ and finite
for $s \in (0, 3)$.

    In order to prove (ii) fix again some $x \in \mathbb{R}^3$.
    It is clear that $\int_\Sigma 1 \mathrm{d}\sigma(y) = \sigma(\Sigma)$
    is finite independent of $x$. 
    Furthermore, since $\Sigma$ is compact there exists $R_1 > 0$ such that
$\Sigma \subset B(0, R_1-1)$.
    If $|x| > R_1$, then $|x - y| > 1$ for all $y \in \Sigma$ and therefore
    \begin{equation*}
      \int_\Sigma |x - y|^{-s} \mathrm{d} \sigma(y) \leq \int_\Sigma
\mathrm{d} \sigma(y)
          = \sigma(\Sigma).
    \end{equation*}

    If $|x| \leq R_1$, we need a slightly more sophisticated estimate which
follows the ideas of \cite[Proposition~A.4]{BEHL16}.
    Define
    \begin{equation*}
      A_n = \left\{ y \in \Sigma: 2^{-n} \leq |x - y|/R_1 < 2^{-n+1}
\right\}, \qquad n = 0, 1, 2, \dots,
    \end{equation*}
    so that $\Sigma = \overline{\bigcup_{n=0}^\infty
A_n}$. Moreover, for $y \in A_n$ we have
    \begin{equation*}
      |x - y|^{-s} \leq R_1^{-s} 2^{s n}
    \end{equation*}
    and hence
    \begin{equation*}
      \begin{split}
        \int_{\Sigma} |x - y|^{-s} \mathrm{d} \sigma(y)
        &= \sum_{n=1}^{\infty} \int_{A_n} |x - y|^{-s} \mathrm{d} \sigma(y)
        \leq \sum_{n=1}^{\infty} R_1^{-s} 2^{s n} \int_{A_n} \mathrm{d}
\sigma(y).
      \end{split}
    \end{equation*}
    Since $\Sigma$ is a smooth and bounded surface there is a constant $k
= k(\Sigma)> 0$ such that
    \begin{equation*}
      \sigma(B(x, \rho) \cap \Sigma) \leq k \rho^2
    \end{equation*}
    independent of $x \in \mathbb{R}^3$ and $\rho > 0$,
    cf. \cite[Chapter II, Example 3]{JW84}.
    Using the fact that $A_n \subset B(x, R_1 \cdot 2^{-n+1})$ it follows
that
    \begin{equation*}
      \begin{split}
        \int_{\Sigma} |x - y|^{-s} \mathrm{d} \sigma(y)
        & \leq \sum_{n=1}^{\infty} k R_1^{-s} 2^{s n} (R_1 \cdot 2^{-n+1})^2
        = 4 k R_1^{2-s} \sum_{n=1}^{\infty}  2^{(s-2) n}.
      \end{split}
    \end{equation*}
    Since $s \in (0, 2)$ the last sum is finite. Therefore, the claim is also true
    in the case $|x| \leq R_1$. The proof of Lemma~\ref{lemma_integral_estimates}~(ii) is complete.
\end{proof}

Finally, by applying the Schur test and the estimates from the previous
lemma, we can show that
integral operators with suitable integral kernels are
bounded and everywhere defined and we get estimates for their operator
norms.
The results are formulated such that they can be applied
directly in the main part of the paper.

\begin{prop} \label{proposition_integral_operators1}
    Let $t: \mathbb{R}^3 \rightarrow \mathbb{C}^{n\times n}$ be measurable 
    and assume that there exist positive constants $\kappa_1, \kappa_2$
and $R$ such that
    \begin{equation*}
      |t(x)| \leq \kappa_1 \begin{cases} |x|^{-2}, &~ |x| < R, \\
                                            e^{-\kappa_2 |x|},& ~ |x|
\geq R, \end{cases}
    \end{equation*}
    for $x \in \mathbb{R}^3 \setminus \{ 0 \}$. Then the operator $T: 
L^2(\mathbb{R}^3; \mathbb{C}^n) \rightarrow
L^2(\mathbb{R}^3; \mathbb{C}^n)$,
    \begin{equation*}
      T f(x) := \int_{\mathbb{R}^3} t(x - y) f(y) \mathrm{d} y, \quad x
\in \mathbb{R}^3, ~f \in L^2(\mathbb{R}^3; \mathbb{C}^n),
    \end{equation*}
    is everywhere defined and bounded with  $\| T \| \leq \kappa_1 K$ for some $K>0$.
\end{prop}
\begin{proof}
    We define for $x \in \mathbb{R}^3 \setminus \{ 0 \}$
    \begin{equation*}
      \tau(x) := \kappa_1 \begin{cases} |x|^{-2}, &~ |x| < R, \\
                                        e^{-\kappa_2 |x|},& ~ |x| \geq R,
\end{cases}
    \end{equation*}
    and $t_1(x, y) = t_2(x, y) := \tau(x - y)$ for $x, y \in \mathbb{R}^3$.
    Then, it follows from Lemma~\ref{lemma_integral_estimates}~(i) that
there exists a constant $K$ such that
    \begin{equation*}
      \int_{\mathbb{R}^3} t_1(x, y) \mathrm{d} x
          = \int_{\mathbb{R}^3} \tau(x - y) \mathrm{d} x \leq \kappa_1 K
    \end{equation*}
    for almost every $y \in \mathbb{R}^3$. Hence, the Schur test
(Proposition~\ref{proposition_Schur_test})
    implies that $T$ is bounded and everywhere defined and that $\| T \|
\leq \kappa_1 K$ holds.
\end{proof}

\begin{prop} \label{proposition_integral_operators2}
    Let $t: \mathbb{R}^3 \rightarrow \mathbb{C}^{n \times n}$
    be measurable and assume that there exist positive constants $\kappa_1, \kappa_2$
and $R$ such that
    \begin{equation*}
      |t(x)| \leq \kappa_1 \begin{cases} |x|^{-2},& ~ |x| < R, \\
                                         e^{-\kappa_2 |x|},& ~ |x| \geq
R, \end{cases}
    \end{equation*}
    for $x \in \mathbb{R}^3 \setminus \{ 0 \}$. Then the operators
    $T_1: L^2(\mathbb{R}^3; \mathbb{C}^n) \rightarrow L^2(\Sigma; 
\mathbb{C}^n)$,
    \begin{equation*}
      T_1 f(x) := \int_{\mathbb{R}^3} t(x - y) f(y) \mathrm{d} y, \quad x
\in \Sigma, ~f \in L^2(\mathbb{R}^3; \mathbb{C}^n),
    \end{equation*}
    and $T_2: L^2(\Sigma; \mathbb{C}^n) \rightarrow L^2(\mathbb{R}^3; 
\mathbb{C}^n)$,
    \begin{equation*}
      T_2 \varphi(x) := \int_\Sigma t(x - y) \varphi(y) \mathrm{d}
\sigma(y), \quad x \in \mathbb{R}^3, ~\varphi \in L^2(\Sigma; \mathbb{C}^n),
    \end{equation*}
    are everywhere defined and bounded with 
    $\| T_1 \|, \| T_2 \| \leq \kappa_1 K$ for some $K>0$.
\end{prop}

\begin{proof}
    We prove the statement for the operator $T_1$, the claim for $T_2$
follows then by taking adjoints.
    Let us define for an $s \in (0, 1)$ and $x \in \mathbb{R}^3 
\setminus \{ 0 \}$
    \begin{equation*}
      \tau_1(x) := \kappa_1 \kappa_3 |x|^{-2+s}
    \end{equation*}
    and
    \begin{equation*}
      \tau_2(x) := \kappa_1 \begin{cases} |x|^{-2-s},&~ |x| < R, \\
                                          e^{-\kappa_2 |x|},&~ |x| \geq R,
                   \end{cases}
    \end{equation*}
    where the constant $\kappa_3$ is chosen such that $e^{-\kappa_2 |x|}
\leq \kappa_3 |x|^{-2+s}$
    for $|x| \geq R$. Set $t_j(x, y) := \tau_j(x - y)$ for
$j \in \{ 1, 2 \}$ and $x \in \Sigma,~ y \in \mathbb{R}^3$,
    and note that the estimate 
    $|t(x-y)|^2 \leq t_1(x, y) t_2(x,y)$ holds for almost all $x, y$.
    By applying Lemma~\ref{lemma_integral_estimates}~(ii) we
see that there is a constant $K_1$ such that 
    \begin{equation*}
      \int_\Sigma t_1(x, y) \mathrm{d} \sigma(x) = \int_\Sigma \tau_1(x -
y) \mathrm{d} \sigma(x) \leq \kappa_1 K_1
    \end{equation*}
    for almost all $y \in \mathbb{R}^3$. Similarly,
Lemma~\ref{lemma_integral_estimates}~(i) implies that 
    \begin{equation*}
      \int_{\mathbb{R}^3} t_2(x, y) \mathrm{d} y = \int_{\mathbb{R}^3}
\tau_2(x - y) \mathrm{d} y \leq \kappa_1 K_2
    \end{equation*}
    is true for almost all $x \in \Sigma$ and a constant $K_2$.
    Therefore, Proposition~\ref{proposition_Schur_test}
    yields the assertions for $T_1$.
\end{proof}

\begin{prop} \label{proposition_integral_operators3}
    Let $t: \mathbb{R}^3 \rightarrow \mathbb{C}^{n \times n}$ be measurable
    and assume that there exists a constant $\kappa > 0$ such that
    \begin{equation*}
      |t(x)| \leq \kappa \big( 1 + |x|^{-1} \big)
    \end{equation*}
    for $x \in \mathbb{R}^3 \setminus \{ 0 \}$. Then, the operator $T: 
L^2(\Sigma; \mathbb{C}^n) \rightarrow L^2(\Sigma; \mathbb{C}^n)$,
    \begin{equation*}
      T \varphi(x) := \int_{\Sigma} t(x - y) \varphi(y) \mathrm{d}
\sigma(y), \quad x \in \Sigma, ~\varphi \in L^2(\Sigma; \mathbb{C}^n),
    \end{equation*}
    is everywhere defined and bounded with 
    $\| T \| \leq \kappa K$ for some $K>0$.
\end{prop}
\begin{proof}
    We define the functions
    \begin{equation*}
      \tau(x) := \kappa \big( 1 + |x|^{-1}\big), \quad x \in \mathbb{R}^3 \setminus \{ 0 \},
    \end{equation*}
    and $t_1(x, y) = t_2(x, y) := \tau(x - y)$ for $x, y \in \Sigma$.
    Lemma~\ref{lemma_integral_estimates}~(ii) shows that there is a
constant $K>0$ such that
    \begin{equation*}
      \int_{\Sigma} t_1(x, y) \mathrm{d} \sigma(x)
      = \int_{\Sigma} \tau(x - y) \mathrm{d} \sigma(x) \leq \kappa K
    \end{equation*}
    for almost every $y \in \Sigma$. Hence
Proposition~\ref{proposition_Schur_test} implies the statement. 
\end{proof}

\end{appendix}

%%%%%%%%%%%%%%%%%%%%%%%%%%%%%%%%
%%%%%%%%%%%%%%%%%%%%%%%%%%%%%%%%
\vskip 0.8cm
\noindent {\bf Acknowledgments.}  
Jussi Behrndt and Markus Holzmann gratefully
acknowledge financial support
by the Austrian Science Fund (FWF): Project P~25162-N26.
Pavel Exner and Vladimir Lotoreichik gratefully
acknowledge financial support
by the Czech Science Foundation (GA\v{C}R): Project 14-06818S.

%%%%%%%%%%%%%%%%%%%%%%%%%%%%%%%%
%%%%%%%%%%%%%%%%%%%%%%%%%%%%%%%%

% *******************************************************************
% *******************************************************************

% *******************************************************************


\begin{thebibliography}{99}
% *******************************************************************
%\normalsize

\bibitem{AGHH05}
S.~Albeverio, F.~Gesztesy, R.~H{\o}egh-Krohn, and H.~Holden,
Solvable Models in Quantum Mechanics. With an Appendix by Pavel Exner. 
2nd ed.
(AMS Chelsea Publishing, Providence, RI, 2005).

\bibitem{AMV14}
N.~Arrizabalaga, A.~Mas, and L.~Vega,
{\it Shell interactions for {D}irac operators},
J. Math. Pures Appl. (9), 102(4): 617--639, 2014.

\bibitem{AMV15}
N.~Arrizabalaga, A.~Mas, and L.~Vega,
{\it Shell interactions for {D}irac operators: on the point spectrum and 
the confinement},
SIAM J. Math. Anal., 47(2): 1044--1069, 2015.

\bibitem{AMV16}
N.~Arrizabalaga, A.~Mas, and L.~Vega,
{\it An {I}soperimetric-{T}ype {I}nequality for {E}lectrostatic {S}hell 
{I}nteractions for {D}irac {O}perators},
Comm. Math. Phys., 344(2): 483--505, 2016.

\bibitem{BEHL16}
J.~Behrndt, P.~Exner, M.~Holzmann, and V.~Lotoreichik,
{\it Approximation of {S}chr\"{o}dinger operators with 
$\delta$-interactions supported on hypersurfaces},
to appear in Math. Nachr.

\bibitem{BL07}
J.~Behrndt and M.~Langer,
{\it Boundary value problems for elliptic partial differential operators 
on bounded domains},
J. Funct. Anal., 243(2): 536--565, 2007.

\bibitem{BL12} J.~Behrndt and M.~Langer, {\it Elliptic operators, 
Dirichlet-to-Neumann maps and quasi boundary triples},
London Math.\ Soc.\ Lecture Note Series 404: 121--160, 2012.

\bibitem{BLL13}
J.~Behrndt, M.~Langer, and V.~Lotoreichik,
{\it Schr\"odinger operators with {$\delta$} and {$\delta'$}-potentials 
supported on hypersurfaces},
Ann. Henri Poincar\'e, 14(2): 385--423, 2013.

\bibitem{BLL13_1}
J.~Behrndt, M.~Langer, and V.~Lotoreichik,
{\it Spectral estimates for resolvent differences of self-adjoint 
elliptic operators},
Integral Equations Operator Theory, 77(1): 1--37, 2013.

\bibitem{BLL13_2}
J.~Behrndt, M.~Langer, and V.~Lotoreichik,
{\it Trace formulae and singular values of resolvent power differences 
of self-adjoint elliptic operators},
J. Lond. Math. Soc. (2), 88(2): 319--337, 2013.

\bibitem{BK13}
G.~Berkolaiko and P.~Kuchment,
Introduction to Quantum Graphs,
American Mathematical Society, Providence, RI, 2013.

\bibitem{BEKS94}
J.~Brasche, P.~Exner, Y.~Kuperin, and P.~\v{S}eba,
{\it Schr\"odinger operators with singular interactions},
J.\ Math.\ Anal.\ Appl. 184: 112--139, 1994.

\bibitem{B76}
V.M.~Bruk,
{\it A certain class of boundary value problems with a spectral 
parameter in the boundary condition},
Mat. Sb. 100 (142): 210--216, 1976.

\bibitem{BGP08}  J.~Br\"uning, V.~Geyler, and K.~Pankrashkin,
{\it Spectra of self-adjoint extensions and applications to solvable 
Schr\"odinger operators},
Rev. Math. Phys. 20: 1--70, 2008.

\bibitem{CMP13}
R.~Carlone, M.~Malamud and A.~Posilicano,
{\it On the spectral theory of {G}esztesy-\v {S}eba realizations of 1-{D} 
{D}irac operators with point interactions on a discrete
set},
J. Differential Equations \textbf{254}(9): 3835--3902, 2013.

\bibitem{DHMS06}
V.\,A.~Derkach, S.~Hassi, M.\,M.~Malamud, and H.~de~Snoo,
{\it Boundary relations and their Weyl families},
Trans. Amer. Math. Soc. 358: 5351--5400, 2006.


\bibitem{DM91} V.\,A.~Derkach and M.\,M.~Malamud,
{\it Generalized resolvents and the boundary value problems for 
Hermitian operators with gaps},
J.\ Funct.\ Anal. 95: 1--95, 1991.

\bibitem{DM95}
V.\,A.~Derkach and M.\,M.~Malamud,
{\it The extension theory of {H}ermitian operators and the moment problem},
J. Math. Sci., 73(2): 141--242, 1995.

\bibitem{DES89}
J.~Dittrich, P.~Exner, and P.~{\v{S}}eba,
{\it Dirac operators with a spherically symmetric {$\delta$}-shell 
interaction},
J. Math. Phys., 30(12): 2875--2882, 1989.

\bibitem{E03} P.~Exner, 
{\it Spectral properties of Schr\"odinger operators with a strongly attractive $\delta$ interaction supported by a surface}, 
Proc. of the NSF Summer Research Conference (Mt. Holyoke 2002); AMS ``Contemporary Mathematics" Series \textbf{339}: 25--36, 2003.

\bibitem{E08}
P.~Exner,
{\it Leaky quantum graphs: a review},
In {\it Analysis on graphs and its applications}, volume~77 of
   Proc. Sympos. Pure Math.: 523--564. Amer. Math. Soc., 
Providence, RI,
   2008.

\bibitem{EK15}
P.~Exner and H.~Kova{\v{r}}{\'{\i}}k,
Quantum Waveguides,
Theoretical and Mathematical Physics. Springer, Cham, 2015.

\bibitem{GHSZ96}
F.~Gesztesy, H.~Holden, B.~Simon and Z.~Zhao,
{\it A trace formula for multidimensional {S}chr\"odinger operators},
J. Funct. Anal. 141(2): 449--465, 1996.

\bibitem{GMZ07}
F.~Gesztesy, M.~Mitrea and M.~Zinchenko,
{\it Variations on a theme of {J}ost and {P}ais},
J. Funct. Anal. 253(2): 399--448, 2007.

\bibitem{GS87}
F.~Gesztesy and P.~{\v{S}}eba,
{\it New analytically solvable models of relativistic point interactions},
Lett. Math. Phys. 13(4): 345--358, 1987.

\bibitem{GK69} I.\,C.~Gohberg and M.\,G.~Kre\u\i{}n,
Introduction to the Theory of Linear Nonselfadjoint Operators,
Transl.\ Math.\ Monogr. 18,
Amer.\ Math.\ Soc., Providence, RI, 1969.

\bibitem{GG91}
V.I.~Gorbachuk and M.L.~Gorbachuk,
Boundary Value Problems for Operator Differential Equations,
Kluwer Academic Publ., Dordrecht, 1991.

\bibitem{JW84}
A.~Jonsson and H.~Wallin.
Function Spaces on Subsets of {${\bf R}^n$}.
Math. Rep., 2(1): xiv+221, 1984.

\bibitem{K75}
A.N.~Ko\u{c}ube\u{\i},
{\it On extensions of symmetric operators and symmetric binary relations},
Mat. Zametki 17: 41--48, 1975.

\bibitem{K95}
T.~Kato,
Perturbation Theory for Linear Operators,
Springer-Verlag, Berlin, 1995.

\bibitem{RS79} 
M.~Reed and B.~Simon, 
Methods of Modern Mathematical Physics III. Scattering Theory, 
Academic Press, New York-London, 1979.

\bibitem{T92}
B.~Thaller,
The {D}irac Equation,
Texts and Monographs in Physics. Springer-Verlag, Berlin, 1992.

\bibitem{WBSB14}
T.~O. Wehling, A.~M. Black-Schaffer, and A.~V. Balatsky.
{\it Dirac materials},
Advances in Physics 63(1): 1--76, 2014.

\bibitem{W00}
J.~Weidmann,
Lineare Operatoren in Hilbertr\"aumen. Teil I.
Teubner, Stuttgart, 2000.

\bibitem{W03}
J.~Weidmann,
Lineare Operatoren in Hilbertr\"aumen. Teil II.
Teubner, Stuttgart, 2003.

\bibitem{Y10}
D.R.~Yafaev,
Mathematical scattering theory. Analytic theory.
American Mathematical Society, Providence, RI, 2010.

% *******************************************************************
\end{thebibliography}
\end{document}